\documentclass[11pt,a4paper,twoside, abstracton
]{scrartcl}
\usepackage{myscrartcl}
\usepackage[T1]{fontenc}

\usepackage{subfigure}
\usepackage{amsfonts} %
\usepackage{amssymb} %
\usepackage{amsbsy} %
\usepackage{pstricks,pst-node,pst-text,pst-3d, pst-plot, pst-grad} %
\usepackage{array}
\usepackage{epsfig} %

\usepackage{definizioni, amssymb, amsbsy, amsmath}

\usepackage{amsthm}

\theoremstyle{plain}
\newcounter{thm}[section]

\newtheorem{theorem}[thm]{Theorem}

\newtheorem{lemma}[thm]{Lemma}
\newtheorem{proposition}[thm]{Proposition}

\newcounter{ass}[section]

\newtheorem{assumption}[ass]{Assumption}

\theoremstyle{definition}
\newcounter{exer}[section]

\newcounter{exam}[section]

\newtheorem{example}[exam]{Example}

\newcounter{defi}[section]

\newtheorem{definition}[defi]{Definition}

\newcounter{rem}[section]

\newtheorem{remark}[rem]{Remark}

\numberwithin{equation}{section}
\title{Bang--bang trajectories with  a double switching time:
  sufficient strong local optimality conditions}
\author{Laura Poggiolini $\qquad$ Marco
  Spadini\thanks{Dipartimento di Matematica Applicata, Universit\`a
    degli Studi di Firenze, via di Santa Marta, 3, I-50139 Firenze
    ({\tt laura.poggiolini@math.unifi.it} and {\tt
      marco.spadini@math.unifi.it})
}}

\date{}

\begin{document}
\maketitle


\begin{abstract}
  This paper gives sufficient conditions for a class of bang-bang
  extremals with multiple switches to be locally optimal in the strong
  topology.  The conditions are the natural generalizations of the
  ones considered in \cite{ASZ02b, Pog06} and \cite{PS04}. We require
  both the {\em strict bang-bang Legendre condition}, and the second
  order conditions for the finite dimensional problem obtained by
  moving the switching times of the reference trajectory.
\end{abstract}

%
%

\pagestyle{myheadings}
\thispagestyle{plain}
\markboth{Laura Poggiolini and Marco Spadini}{Strong local optimality}

\section{Introduction}
We consider a Mayer problem, where the control functions are bounded
and enter linearly in the dynamics.
\begin{subequations}
  \label{eq:minipb}
  \begin{align}
    \text{minimize }\; & C(\xi, u) := c_0(\xi(0)) + c_f(\xi(T)) \;
    \label{eq:1} \\
    \text{ subject to } \; & \dot\xi(t) = f_0 (\xi(t)) + \dsum_{s=1}^m u_s f_s(\xi(t)) \label{eq:2} \\
    & \xi(0) \in N_0 \, , \quad \xi(T) \in N_f \label{eq:3}  \\
    & u = (u_1, \ldots, u_m) \in L^\infty([0, T],
    [-1,1]^m).\label{eq:4}
  \end{align}
\end{subequations}
Here $T>0$ is given, the state space is a $n$-dimensional manifold $M$,
$N_0$ and $N_f$ are smooth sub-manifolds of $M$. The vector fields
$f_0, f_1,\dots ,f_m$ and the functions $c_0 $, $c_f$ are $C^2$ on
$M$, $N_0$ and $N_f$, respectively.

We aim at giving second order sufficient conditions for a
{\em reference bang-bang extremal couple $(\wh\xi, \wh u)$ to be a
  local optimizer in the strong topology}; the strong topology being
the one induced by $C([0,T], M)$ on the set of admissible
trajectories, regardless of any distance of the associated controls.
Therefore, optimality is with respect to neighboring trajectories,
independently of the values of the associated controls.
In particular, if the extremal is abnormal, we prove that $\wh\xi$ is
isolated among admissible trajectories.

We recall that a control $\wh u$ (a trajectory $\wh \xi$) is bang-bang
if there is a finite number of switching times $\,0 < \hat t_1 < \dots
< \hat t_r < T\,$ such that each component $\wh u_i$ of the reference
control $\wh u$ is constantly either $-1$ or $1$ on each interval
$(\hat t_k,\hat t_{k+1})$.  A switching time $\hat t_k$ is called {\em
  simple} if only one control component changes value at $\hat t_k$,
while it is called {\em multiple} if at least two control components
change value.

%
Second order conditions for the optimality of a bang-bang extremal
with simple switches only are given in \cite{ASZ02b, MO03, Pog06,
  PS04} and references therein, while in \cite{Sar97} the author gives
sufficient conditions, in the case of the minimum time problem, for
$L^1$-local optimality - an intermediate condition between strong
and local optimality - of a bang-bang extremal having both simple
and multiple switches with the extra assumption that the Lie brackets
of the switching vector fields is annihilated by the adjoint covector.

All the above cited papers require regularity assumptions on the
switches (see the subsequent Assumptions \ref{ass:1}, \ref{ass:2} and
\ref{ass:3} which are the natural strengthening of necessary
conditions) and the positivity of a suitable second variation. 

Here we consider the problem of strong local optimality in the case of
a Mayer problem, when at most one double switch occurs, but there are
finitely many simple ones and no commutativity assumptions on the involved
vector fields.  More precisely we extend the conditions in
\cite{ASZ02b, Pog06, PS04} by requiring the sufficient second order
conditions for the finite dimensional sub-problems that are obtained
by allowing the switching times to move. The addition of a double switch 
is not a trivial extension of the known single-switch cases. In fact, as
explained in Section 2.2, any perturbation of the switching time (of 
a double switch) of the components of $\hat u$ generically creates
two simple switches, that is it a \emph{bang} arc is generated. On
the contrary, the small perturbations of a single switch do not change 
the structure of the reference control.

We believe that the techniques employed here could be extended to the
more general case when there are more than one double switch. However,
such an extension may not be straightforward as the technical and
notational complexities grow quickly with the number of double switches.

Preliminary results were
given in \cite{PS06c}, where the authors exploit a study case and in 
\cite{PSp08} that deals with a Bolza problem in the
so-called {\em non-degenerate case}. Also stability analysis under
parameter perturbations for this kind of bang-bang extremals was
studied in \cite{FPS09}.

We point out that, while in
the case of simple switches the only variables are the switching
times, each time a double switch occurs one has to consider the two
possible combinations of the switching controls. The investigation of
the invertibility of the involved Lipschitz continuous, piecewise
$C^1$ operators has been done via some topological methods described
in the Appendix, or via Clarke's implicit function theorem (see
\cite[Thm 7.1.1.]{Cla83}) in some particular degenerate case.

The paper is organized as follows: Section 2.1 introduces the notation 
and the regularity hypotheses that are assumed through the paper. In
Section 2.2, where our main result Theorem \ref{thm:main} is stated, 
we introduce a finite dimensional subproblem of \eqref{eq:minipb} and its 
``second variations'' (indeed this subproblem is $C^{1,1}$ but not $C^2$
so that the classical ``second variation'' is not well defined). The 
essence of the paper will be to show that the sufficient conditions for 
the optimality of an extremum of this subproblem are actually sufficient 
also for the optimality of the reference pair $(\hat\xi,\hat u)$ in 
problem \eqref{eq:minipb}. In Section 3 we briefly describe the Hamiltonian 
methods the proof is based upon. Section 4 contains the maximized 
Halmiltonian of the control system and its flow. In Section 5, we write 
the ``second variations'' of the finite-dimensional subproblem and study 
their sign on appropriate spaces. Section 6 is the heart of the paper and 
constitutes its more original contribution; here  we prove that the the 
projection onto a neighborhood of the graph of $\hat\xi$ in $\R\times M$ 
of the maximized flow defined in Section 4 is invertible (which is
necessary for our Hamiltonian methods to work). Section 7
contains the conclusion of the proof of  Theorem \ref{thm:main}. In the
Appendix we treat from an abstract viewpoint the problem, raised in
Section 6, of local invertibility of a piecewise $C^1$ function.

\section{The result}

The result is based on some regularity assumption on the vector fields
associated to the problem and on a second order condition for a finite
dimensional sub-problem. 
The regularity Assumptions \ref{ass:2} and \ref{ass:3} are natural,
since we look for sufficient conditions. In fact Pontryagin Maximum
Principle yields the necessity of the same inequalities but in weak
form.

\subsection{Notation and regularity}
\label{sec:notreg}

We assume we are given an admissible reference couple $\big( \wh\xi,
\wh u \big)$ satisfying Pontryagin Maximum Principle (PMP) with
adjoint covector $\wh\lambda$ and that the reference control $\wh u$
is bang-bang with switching times $\wh t_1, \ldots, \wh t_r$ such
that only two kinds of switchings appear:
\begin{itemize}
\item $\wh t_i$ is a {\em simple switching time} i.e.~only one of the
  control components $\wh u_1$, \ldots, $\wh u_m$ switches at time
  $\wh t_i$;
\item $\wh t_i$ is a {\em double switching time} i.e.~exactly two of
  the control components $\wh u_1$, \ldots, $\wh u_m$ switch at time
  $\wh t_i$.
\end{itemize}
We assume that there is just one double switching time, which we
denote by $\hat\tau $. Without loss of generality we may assume that
the control components switching at time $\hat\tau$ are $\wh u_1$ and
$\wh u_2$ and that they both switch from the value $-1$ to the value
$+1$, i.e.
\[
\displaystyle\lim_{t \to \hat\tau-}\wh u_\nu = -1 \, \quad
\displaystyle\lim_{t \to \hat\tau+}\wh u_\nu = 1 \, \quad \nu=1, 2.
\]
In the interval $(0, \hat\tau)$, $J_0$ simple switches occur (if no
simple switch occurs in $(0 , \hat\tau)$, then $J_0 = 0$), and $J_1$
simple switches occur in the interval $(\hat\tau, T)$ (if no simple
switch occurs in $(\hat\tau, T)$, then $J_1 = 0$). We denote the
simple switching times occurring before the double one by
$\hat\theta_{0j}$, $j=1, \ldots, J_0$, and by $\hat\theta_{1j}$, $j=1,
\ldots, J_1$ the simple switching times occurring afterwards. In order
to simplify the notation, we also define $\hat\theta_{00} := 0$,
$\hat\theta_{0, J_0 + 1} := \hat\theta_{10} := \hat\tau$,
$\hat\theta_{1, J_1 + 1} := T$, i.e.~we have
\begin{equation*}
  \hat\theta_{00} := 0 < \hat\theta_{01} < \ldots 
  < \hat\theta_{0 J_0} < \hat\tau := \hat\theta_{0, J_0 + 1}
  := \hat\theta_{10} < \hat\theta_{11} 
  < \ldots < \hat\theta_{1 J_1} < T := \hat\theta_{1, J_1 + 1}.
\end{equation*}
   \begin{figure}[h!]
    \centering
      \begin{pspicture}(0,-1)(10.5,0) \psline{->}(0,0)(10.5,0)
        \psline{-}(0,-.2)(0,.2) \uput[d](0,-0.25){$0$}
        \psline{-}(.5,-0.2)(.5,0.2)
        \uput[d](.5,-0.25){$\wh\theta_{01}$}
        \uput[d](1.75,-0.25){$\dots$} \psline{-}(3,-0.2)(3,0.2)
        \uput[d](3,-0.25){$\wh\theta_{0J_0}$}
        \psline{-}(4,-0.2)(4,0.2)
        \uput[d](4,-0.25){$\wh\tau$}
        \psline{-}(5.5,-0.2)(5.5,0.2)
        \uput[d](5.53,-0.25){$\wh\theta_{10}$}
        \uput[d](7.25,-0.25){$\dots$} \psline{-}(9,-0.2)(9,0.2)
        \uput[d](9,-0.25){$\wh\theta_{1J_1}$}
        \psline{-}(10,-0.2)(10,0.2) \uput[d](10,-0.25){$T$}
      \end{pspicture}
\caption{The sequence of switching times}
   \end{figure}
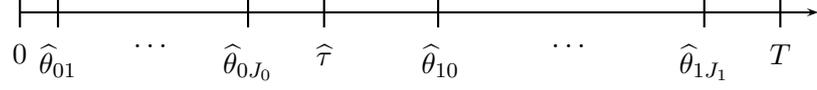

We shall use some basic tools and notation from differential
geometry. For any sub-manifold $N$ of $M$, and any $x \in N$, $T_xN$
and $T^*_xN$ denote the tangent space to $N$ at $x$ and the cotangent
space to $N$ at $x$, respectively while $T^*N$ denotes the cotangent
bundle.  For any $w \in T^*_xM$ and any $\dx \in T_xM$,
$\scal{w}{\dx}$ denotes the duality product between a form and a
tangent vector.

$\pi\colon T^*M \to M$ denotes the canonical projection from the
tangent bundle onto the base manifold $M$. In coordinates $\ell := (p,
x)$:
\[
\pi \colon \ell = (p, x) \in T^*M \mapsto x \in M.
\]
Throughout the paper, for any vector field $f \colon x \in M \mapsto
f(x) \in T_x M$, we shall denote the associated Hamiltonian obtained
by lifting $f$ to $T^*M$ by the corresponding capital letter, i.e.
\[
F \colon \ell \in T^*M \mapsto \scal{\ell}{f(\pi\ell)}\in \R,
\]
and $\vF$ will denote the Hamiltonian vector field associated to $F$.
In particular for any $s =0, 1, \ldots, m$ $F_s(\ell) :=
\scal{\ell}{f_s(\pi\ell)}$ is the Hamiltonian associated to the drift
($s=0$) and to the controlled vector fields of system \eqref{eq:2}.

If $f, g \colon x \in M \mapsto f(x) \in T_x M$, are differentiable
vector fields, we denote their Lie bracket as $[f, g]$:
\[ [f, g](x) := \uD  g(x) \, f(x) - \uD  f(x) \, g(x)
\]
The canonical symplectic two-form between two Hamiltonian vector
fields $\vF{}$ and $\vG{}$ at a
point $\ell$ is denoted as $\dueforma{\vF{}}{\vG{}}(\ell)$. In
coordinates $\ell := (p, x)$:
\[
\dueforma{\vF{}}{\vG{}}(\ell) := -\scal{p\uD  g(x)}{f(x)} +
\scal{p\uD  f(x)}{g(x)} .
\]
For any $m$-tuple $u = (u_1, \ldots, u_m) \in \R^m$ let us denote the
control-dependent Hamiltonian
\[
h_u\colon\ell \in T^*M \mapsto \scal{\ell}{f_0(\pi\ell) +
  \dsum_{s=1}^m u_s f_s(\pi\ell)} \in \R.
\]
Let $\wh f_t$ and $\wh F_t$ be the reference vector field and the
reference Hamiltonian, respectively:
\[
\wh f_t (x) := f_0(x) + \dsum_{s=1}^m \wh u_s(t) f_s(x) \, , \quad \wh
F_t (\ell) := \scal{\ell}{\wh f_t(\pi\ell)} = h_{\wh u(t)} (\ell)
\]
and let
\[
H(\ell) := \max \left\{ h_u(\ell ) \colon u \in [-1, 1]^m \right\}
\]
be the maximized Hamiltonian of the control system.  Also, let $\wh
x_0 := \wh\xi(0)$, $\wh x_d := \wh\xi(\hat\tau)$ and $\wh x_f :=
\wh\xi(T)$.

The reference flow, that is the flow associated to $\wh f_t$, is
defined on the whole interval $[0, T]$ at least in a neighborhood of
$\wh x_0$. We denote it as
\[
\wh S \colon (t, x) \mapsto \wh S_t(x).
\]
Thus, in our situation PMP reads as follows:\\
{\em There exist $p_0 \in \{0, 1\}$ and an absolutely continuous
  function $\wh\lambda\colon [0, T] \to T^*M$ such that}
\begin{align}
  & (p_0, \wh\lambda(0)) \neq (0,0) \displaybreak[0] \\
  & \pi\wh\lambda(t) = \wh\xi(t) \qquad \forall t \in [0, T] \notag
  \displaybreak[0] \\
  \qquad
  &\dot{\wh\lambda}(t) = \overrightarrow{\wh F}_t(\wh\lambda(t)) \qquad \qo t \in [0, T], \notag  \displaybreak[0] \\
  & \wh\lambda(0)\vert_{T_{\wh x_0} N_0} = p_0\ud c_0(\wh x_0), \qquad
  \wh\lambda(T)\vert_{T_{\wh x_f} N_f} = -  p_0\ud c_f(\wh x_f)  \label{eq:trasv} \displaybreak[0] \\
  & \wh F_t (\wh\lambda(t)) = H(\wh\lambda(t)) \quad \qo t \in [0, T].
  \label{eq:maxi}
\end{align}
We shall denote $\lo := \wh\lambda(0)$ and $\lf := \wh\lambda(T)$.

Maximality condition \eqref{eq:maxi} implies $\wh u_s(t)
F_s(\wh\lambda(t)) = \wh u_s(t) \scal{\wh\lambda(t)}{f_s(\wh\xi(t))}
\geq 0$ for any $t\in [0, T]$ and any $s=1, \ldots, m$.  We assume the
following regularity condition holds:
\begin{assumption}[Regularity] \label{ass:1} Let $s \in \{1, \ldots, m\}$. If
  $t$ is not a switching time for the control component $\wh u_s$,
  then
  \begin{equation}
    \label{eq:reg}
    u_s(t) F_s(\wh\lambda(t))  =  \wh u_s(t) \scal{\wh\lambda(t)}{f_s(\wh\xi(t))} > 0.
  \end{equation}
\end{assumption}
In terms of the switching functions $\sigma_s \colon t \in [0, T]
\mapsto F_s \circ \wh\lambda(t) \in \R$, $s = 1, \dots, m$ Assumption
\ref{ass:1} means $ \hat u_s(t) = \sgn{\sigma_s(t)} $ whenever $t$ is
not a switching time of the reference control component $\hat u_s$.

Notice that Assumption \ref{ass:1} implies that
$\operatorname{argmax}\{h_u(\wh\lambda(t))\colon u \in [-1, 1]^m \} =
\wh u(t)$ for any $t$ that is not a switching time.

Let
\begin{equation*}
  k_{ij} := \wh f_t \vert_{(\wh \theta_{ij}, \hat\theta_{i, j+1}) } ,
  \quad j =0, \ldots, J_i , \ i=0,1,
\end{equation*}
be the restrictions of $\wh f_t$ to each of the time intervals where
the reference control $\wh u$ is constant and let $K_{ij}(\ell)
\dfrac{}{}:= \scal{\ell}{k_{ij}(\pi\ell)}$ be the associated
Hamiltonian.  Then, from maximality condition \eqref{eq:maxi} we get
\[
\left. \dfrac{\ud}{\ud t}\left( K_{i j} - K_{i, j-1}
  \right)\circ\wh\lambda(t)\right\vert_{t = \hat\theta_{ij}} \geq 0
\]
for any $i = 0, 1, \quad j = 1, \ldots, J_i$, i.e.~if $\hat u_{s(ij)}$
is the control component switching at time $\hat\theta_{ij}$ and
$\Delta_{ij} \in \{-2, 2 \}$ is its jump, then
\begin{equation*}
  \left. \ddt \Delta_{ij}\sigma_{s(ij)}(t)\right\vert_{t = \hat\theta_{ij}}  \geq 0
\end{equation*}
We assume that the strong inequality holds at each simple switching
time $\hat\theta_{ij}$:
\begin{assumption} \label{ass:2}
  \begin{equation}\label{eq:leg1a}
    \left. \dfrac{\ud}{\ud t}\left( 
        K_{ij} - K_{i, j-1}
      \right)\circ\wh\lambda(t)\right\vert_{t = \hat\theta_{ij}} > 0 
    \qquad i = 0,1, \quad j = 1, \ldots, J_i .
  \end{equation}
\end{assumption}
Assumption \ref{ass:2} is known as the {\sc Strong bang-bang Legendre
  condition for simple switching times}.

In geometric terms Assumption \ref{ass:2} means that at time $t =
\hat\theta_{ij}$ the trajectory $t \mapsto \widehat\lambda(t)$ crosses
transversally the hypersurface of $T^*M$ defined by $K_{ij} = K_{i,
  j-1}$, i.e.~by the zero level set of $F_{s(ij)}$.

\begin{figure}[h!]
 \centering
\begin{pspicture}(0,0)(7,4.5)
 \psbezier(0,3.5)(1,3.2)(1.8,3)(3,2.5)
 \psbezier(3,2.5)(4.2,2)(4.4,1)(4.5,0.5)
 \uput[d](4.5,0.5){\small $K_{ij}=K_{i,j-1}$} 
 \psbezier[linecolor=gray,linewidth=2pt](2,0)(2.2,1)(2.5,2)(3,2.5)
  \psline{->}(3,2.5)(4.5,4)
  \uput[l](4.6,4.25){\small $\vK_{i,j-1}\big(\wh\lambda(\wh\theta_{ij})\big)$}
 \psbezier[linecolor=gray,linewidth=2pt](3,2.5)(4,3.1)(5,3.2)(6.5,3)
  \uput[dr](5,3){\small $\wh\lambda(t)$}
  \psline{->}(3,2.5)(5,3.7)
  \uput[r](5,3.7){\small $\vK_{ij}\big(\wh\lambda(\wh\theta_{ij})\big)$}
  \psdot(3,2.5)
  \uput[l](2.9,2.4){\small $\wh\lambda(\wh\theta_{ij})$}
\end{pspicture}
\caption{Behaviour at a simple switching time}
\end{figure}
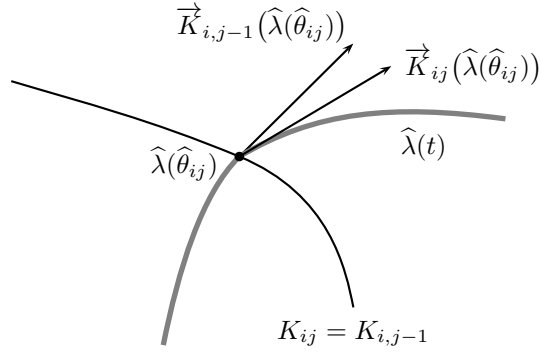

As already said, without any loss of generality we can assume that the
double switching time involves the first two components, $\hat u_1$
and $\hat u_2$ of the reference control $\hat u$ and that they both
switch from $-1$ to $+1$, so that
\[
k_{10} = k_{0J_0} + 2 f_1 + 2f_2.
\]
Define the new vector fields
\[
k_\nu := k_{0J_0} + 2 f_\nu \, , \quad \nu = 1, 2,
\]
with associated Hamiltonians $K_\nu(\ell) := \scal{\ell}{ k_\nu
  (\pi\ell)}$.  Then, from maximality condition \eqref{eq:maxi} we get
\begin{equation*}
  \begin{split}
    & \left.\dfrac{\ud}{\ud t} 2 \, \sigma_\nu (t)\right\vert_{t =
      \hat\tau - } 
    \!\!\!\! = \left.\dfrac{\ud}{\ud t} 2 \, F_\nu
      \circ\wh\lambda(t)\right\vert_{t = \hat\tau - } \!\!\!\! =
    \left.\dfrac{\ud}{\ud t}\left( K_\nu -
        K_{0J_0}\right)\circ\wh\lambda(t)\right\vert_{t = \hat\tau - }
    \!\!\!\! \geq 0, \\
    & \left.\dfrac{\ud}{\ud t} 2 \, \sigma_\nu (t)\right\vert_{t =
      \hat\tau + } 
    \!\!\!\! = \left.\dfrac{\ud}{\ud t} 2 \, F_\nu
      \circ\wh\lambda(t)\right\vert_{t = \hat\tau + } \!\!\!\! = \left.
      \dfrac{\ud}{\ud t}\left( K_{1 0} - K_\nu
      \right)\circ\wh\lambda(t) \right\vert_{t = \hat\tau + } \!\!\!\!
    \geq 0,
  \end{split} \qquad \nu =1, 2 .
\end{equation*}
We assume that the strict inequalities hold:
\begin{assumption}\label{ass:3}
  \begin{equation}\label{eq:leg2a}
    \begin{split}
      & \left.\dfrac{\ud}{\ud t}\left( K_\nu -
          K_{0J_0}\right)\circ\wh\lambda(t)\right\vert_{t = \hat\tau -
      } \!\!\!\! > 0, \qquad \left.  \dfrac{\ud}{\ud t}\left( K_{1 0}
          - K_\nu \right)\circ\wh\lambda(t) \right\vert_{t = \hat\tau
        + } \!\!\!\! > 0,
    \end{split} \quad \nu =1, 2.
  \end{equation}
\end{assumption}
Assumption \ref{ass:3} means that at time $\hat\tau$ the flow arrives
the hypersurfaces $F_1 = 0$ and $F_2 = 0$  with transversal velocity
$\vK_{0J_0}$ and leaves with velocity $\vK_{10}$ which is again
transversal to both the hypersurfaces.
We shall call Assumption \ref{ass:3} the {\sc Strong bang-bang
  Legendre condition for double switching times}.

\begin{figure}[h!]
 \centering
\begin{pspicture}(0,0.5)(8,6.4)
 \pscurve(0,0.5)(1,2)(2.5,3.8)(6,5.7)(7.5,6)
 \pscurve(4.5,0.5)(4,1)(2.5,3.8)(1.8,5.5)(1.7,6.2)
 \psbezier[linecolor=gray,linewidth=2pt](0,6)(1,5.3)(1.5,4.8)(2.5,3.8)
 \psbezier[linecolor=gray,linewidth=2pt](2.5,3.8)(4,4)(5,4.4)(6.5,5)
 \psdot[linewidth=2pt](2.5,3.8)
 \psline{->}(2.5,3.8)(4,2.3)
 \psline{->}(2.5,3.8)(4.5,4.06)
 \uput[r](4.1,1){\small $F_1=0$}
 \uput[r](0.5,1){\small $F_2=0$}
 \uput[l](2.4,3.8){\small $\wh\lambda(\wh\tau)$}
 \uput[r](6.5,5){\small $\wh\lambda(t)$}
 \uput[dr](4,2.3){\small $\vK_{0J_0}\big(\wh\lambda(\wh\tau)\big)$}
 \uput[dr](4.5,4.06){\small $\vK_{10}\big(\wh\lambda(\wh\tau)\big)$}
\end{pspicture}
\caption{Behaviour at the double switching time}
\end{figure}
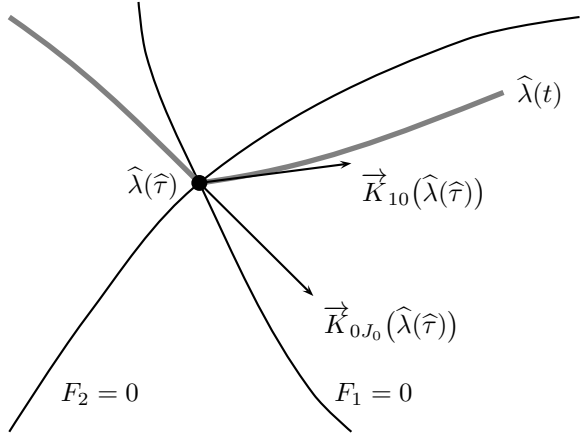

Equivalently, conditions \eqref{eq:leg1a} and \eqref{eq:leg2a} can be
expressed in terms of the Lie brackets of vector fields or in terms of
the canonical symplectic structure $\dueforma{\cdot}{\cdot}$ on
$T^*M$:
\begin{proposition}
  Assumption \ref{ass:2} is equivalent to
  \begin{equation} \label{eq:leg1b}
    \scal{\wh\lambda(\hat\theta_{ij})}{\left[k_{i, j-1}, k_{ij}
      \right] (\wh\xi(\hat\theta_{ij}))} = \dueforma{\vK_{i,
        j-1}}{\vK_{ij}}(\wh\lambda(\hat\theta_{ij})) > 0
  \end{equation}
  for any $ i = 0,1, \quad j = 1, \ldots, J_i$. \\%
  Assumption \ref{ass:3} is equivalent to
  \begin{equation}\label{eq:leg2b}
    \begin{split}
      & \scal{\wh\lambda(\hat\tau)}{\left[ k_{0J_0}, k_\nu \right]
        (\wh x_d)} =
      \dueforma{\vK_{0J_0}}{\vK_\nu}(\wh\lambda(\hat\tau))  > 0, \\
      & \scal{\wh\lambda(\hat\tau)}{\left[ k_\nu, k_{10} \right] (\wh
        x_d)} = \dueforma{\vK_\nu}{\vK_{10}}(\wh\lambda(\hat\tau)) > 0
    \end{split}
    \quad  \nu= 1,2 . 
  \end{equation}
\end{proposition}
In what follows we shall also need to reformulate Assumptions
\ref{ass:2} and \ref{ass:3} in terms of the pull-backs along the
reference flow of the vector fields $k_{ij}$ and $k_\nu$. Define
\begin{equation*}
  g_{ij}(x) 
  := \wh S_{\hat\theta_{ij}\, *}^{-1} k_{ij}  \circ \wh
  S_{\hat\theta_{ij}}(x),
  \quad
  h_\nu(x) := \wh S_{\hat\tau\, *}^{-1} k_\nu \circ \wh S_{\hat\tau}(x)
\end{equation*}
and let $G_{ij}$, $H_\nu$ be the associated Hamiltonians. We can
restate Assumptions \ref{ass:2} and \ref{ass:3} as follows:
\begin{proposition}
  Assumption \ref{ass:2} is equivalent to
  \begin{equation} \label{eq:leg1c} \scal{\lo}{\left[g_{i, j-1},
        g_{ij} \right] (\wh x_0)} = \dueforma{\vG_{i,
        j-1}}{\vG_{ij}}(\lo) > 0
  \end{equation}
  for any $ i = 0,1, \quad j = 1, \ldots, J_i$.\\ %
  Assumption \ref{ass:3} is equivalent to
  \begin{equation}\label{eq:leg2c}
    \begin{split}
      & \scal{\lo}{\left[ g_{0J_0}, h_\nu \right] (\wh x_0)} =
      \dueforma{\vG_{0J_0}}{\vH_\nu}(\lo)  > 0, \\
      & \scal{\lo}{\left[ h_\nu, g_{10} \right] (\wh x_0)} =
      \dueforma{\vH_\nu}{\vG_{10}}(\lo) > 0
    \end{split}
    \quad  \nu= 1,2 . 
  \end{equation}
\end{proposition}

\subsection{The finite dimensional sub-problem}
\label{sec:findim}

By allowing the switching times of the reference control function to
move we can define a finite dimensional sub-problem of the given one.
In doing so we must distinguish between the simple switching times and
the double switching time. Moving a simple switching time
$\hat\theta_{i j}$ to time $\theta_{i j} := \hat\theta_{i j} +
\delta_{i j}$ amounts to using the values $\left. \wh u
\right\vert_{\left(\hat\theta_{i, j-1}, \hat\theta_{ij}\right)}$ and
$\left. \wh u \right\vert_{\left(\hat\theta_{ij}, \hat\theta_{i,
      j+1}\right)}$ of the reference control in the time intervals
$\big( \hat\theta_{i, j-1}, \theta_{i j}\big)$ and $\big( \theta_{i\,
  j}, \hat\theta_{i, j+1}\big)$, respectively. On the other hand, when
we move the double switching time $\hat\tau$ we change the switching
time of two different components of the reference control and we must
allow for each of them to change its switching time independently of
the other. This means that between the values of $\left. \wh u
\right\vert_{\left(\hat\theta_{0J_0}, \hat\tau\right)}$ and
$\left. \wh u \right\vert_{\left(\hat\tau, \hat\theta_{0 1}\right)}$
we introduce a value of the control which is not assumed by the
reference one - at least in a neighborhood of $\hat\tau$ - and which
may assume two different values according to which component switches
first between the two available ones. Let $\tau_\nu := \hat\tau +
\e_\nu$, $\nu = 1,2$.  We move the switching time of the first control
component $\wh u_1$ from $\hat\tau$ to $\tau_1 := \hat\tau + \e_1$,
and the switching time of the second control component $\wh u_2$ from
$\hat\tau$ to $\tau_2 := \hat\tau + \e_2$.

Inspired by \cite{ASZ02b}, let us introduce $C^2$ functions $\alpha
\colon M \to \R$ and $\beta \colon M \to \R$ such that $\alpha|_{N_0}
= p_0 c_0$, $\ud\alpha(\wh x_0) = \lo$ and $\beta|_{N_f} = p_0 c_f$,
$\ud\beta(\wh x_f) = - \lf$.

Define $\theta_{ij} := \hat\theta_{ij} + \delta_{ij}$, $\; j=1, \ldots
, J_i$, $i=0,1$; $\theta_{0,J_0 +1} := \min\{ \tau_\nu, \; \nu =1, 2
\} $, $\theta_{10} := \max\{ \tau_\nu, \; \nu =1, 2 \} $, $\theta_{00}
:= 0$ and $\theta_{1, J_1 +1} := T$.  We have a finite-dimensional
sub-problem (FP) given by
\begin{subequations}
  \label{eq:sottopb}
  \begin{align}
    \label{eq:subcost}\tag{FPa}
    \text{minimize } \; & \alpha(\xi(0)) + \beta(\xi(T)) \\
    \label{eq:sub}\tag{FPb} \text{subject to }\;
    & \dot\xi(t) =
    \begin{cases}
      k_{0j}(\xi(t)) & t \in (\theta_{0j}, \theta_{0, j+1}) \quad j= 0, \ldots, J_0, \\
      k_\nu(\xi(t)) & t \in (\theta_{0, J_0 +1}, \theta_{10}), \\
      k_{1j}(\xi(t)) & t \in (\theta_{1j}, \theta_{1, j+1}) \quad j=
      0, \ldots, J_1
    \end{cases} \displaybreak[0]\\
    \text{and } \; & \xi(0) \in N_0 , \quad \xi(T) \in
    N_f.\label{eq:ini}\tag{FPc} \displaybreak[0]\\
    \text{where } \; &
    \theta_{00} = 0,  \qquad \theta_{1,J_1 +1 } = T \tag{FPd}  \displaybreak[0]\\
    & \theta_{ij} = \hat\theta_{ij} + \delta_{ij}, \quad i= 0, 1,
    \quad j =1, \ldots, J_i,  \tag{FPe} \displaybreak[0]\\
    & \theta_{0,J_0 +1} := \hat\tau + \min\{ \ep_1, \ \ep_2 \}, 
    \quad \theta_{10} := \hat\tau + \max\{ \ep_1, \ \ep_2 \} \tag{FPf} \\
    \text{and } \; & \begin{cases}
      \nu = 1 \quad & \text{if } \ep_1 \leq \ep_2, \\
      \nu = 2 \quad & \text{if } \ep_2 \leq \ep_1.
    \end{cases} \label{eq:nu}\tag{FPg}
  \end{align}
\end{subequations}

\begin{figure}[ht!]
  \centering
  \begin{pspicture}(-0.8,-1)(10.5,0.5) \psline{->}(0,0)(10.5,0)
    \uput[l](0,0){\scriptsize $(\ep_1\leq\ep_2)$}
    \psline{-}(0,-.2)(0,.2) \uput[u](0,0.25){\small $0$}
    \psline{-}(.5,0)(.5,0.2) \uput[u](.5,0.2 ){\small
      $\wh\theta_{01}$} \psline[linestyle=dotted]{-}(0,0.1)(.5,0.1)
    \psline{-}(.75,0)(.75,-.2) \uput[d](.75,-0.1){\small
      $\theta_{01}$} \psline[linestyle=dotted]{-}(0,-0.1)(.75,-0.1)
    \uput[u](1.75,0.2){$\dots$} \psline{-}(3,0)(3,0.2)
    \uput[d](1.6,-0.1){$\dots$ } \uput[u](3,0.2){\small
      $\wh\theta_{0J_0}$} \psline{-}(2.7,0)(2.7,-0.2)
    \uput[d](2.7,-0.1){\small $\theta_{0J_0}$ } \psline{-}(5,0)(5,0.2)
    \uput[u](5,0.2){\small $\wh\tau$} \psline[linestyle=dashed,
    dash=3mm 1mm]{-}(3,0.1)(5,0.1)
    \psline{-}(4,0)(4,-0.2) \uput[d](4,-0.1){\small $\tau_1$}
    \psline[linestyle=dashed, dash=3mm 1mm]{-}(2.7,-0.1)(4,-0.1)
    \psline[linecolor=gray]{<->}(4,-0.1)(5.7,-0.1)
    \psline{-}(5.7,0)(5.7,-0.2) \uput[d](5.7,-0.1){\small $\tau_2$}
    \uput[d](4.85,-0.45){$f_{0J_0} + 2 f_1$}
    \psline{-}(6.5,0)(6.5,0.2) \uput[u](6.5,0.2){\small
      $\wh\theta_{11}$} \psline[linestyle=dashed, dash=1.5mm
    0.5mm]{-}(5,0.1)(6.5,0.1) \psline{-}(6.8,0)(6.8,-0.2)
    \uput[d](6.8,-0.1){\small $\theta_{11}$} \psline[linestyle=dashed,
    dash=1.5mm 0.5mm]{-}(5.7,-0.1)(6.8,-0.1)
    \uput[u](7.75,0.2){$\dots$} \psline{-}(9,0)(9,0.2)
    \uput[u](9,0.2){\small $\wh\theta_{1J_1}$}
    \psline{-}(8.5,0)(8.5,-0.2) \uput[d](7.65,-0.1){$\dots$}
    \uput[d](8.5, -0.1){\small $\theta_{1J_1}$}
    \psline{-}(10,-0.2)(10,0.2) \uput[u](10,0.2){\small $T$}
    \psline[linestyle=dashed, dash=2mm 0.5mm]{-}(9,0.1)(10,0.1)
    \psline[linestyle=dashed, dash=2mm 0.5mm]{-}(8.5,-0.1)(10,-0.1)
  \end{pspicture}
  \begin{pspicture}(-0.8,-1)(10.5,0.5) \psline{->}(0,0)(10.5,0)
    \uput[l](0,0){\scriptsize $(\ep_2\leq\ep_1)$}
    \psline{-}(0,-.2)(0,.2)
    \psline{-}(.5,0)(.5,0.2)
    \psline[linestyle=dotted]{-}(0,0.1)(.5,0.1)
    \psline{-}(.75,0)(.75,-.2) \uput[d](.75,-0.1){\small
      $\theta_{01}$} \psline[linestyle=dotted]{-}(0,-0.1)(.75,-0.1)
    \psline{-}(3,0)(3,0.2)
    \psline{-}(2.7,0)(2.7,-0.2) \uput[d](1.6,-0.1){$\dots$ }
    \uput[d](2.7,-0.1){\small $\theta_{0J_0}$ } \psline{-}(5,0)(5,0.2)
    \psline[linestyle=dashed, dash=3mm 1mm]{-}(3,0.1)(5,0.1)
    \psline{-}(4,0)(4,-0.2) \uput[d](4,-0.1){\small $\tau_2$}
    \psline[linestyle=dashed, dash=3mm 1mm]{-}(2.7,-0.1)(4,-0.1)
    \psline{-}(5.7,0)(5.7,-0.2) \uput[d](5.7,-0.1){\small $\tau_1$}
    \psline[linecolor=gray]{<->}(4,-0.1)(5.7,-0.1)
    \uput[d](4.85,-0.45){$f_{0J_0} + 2 f_2$}
    \psline{-}(6.5,0)(6.5,0.2)
    \psline[linestyle=dashed, dash=1.5mm 0.5mm]{-}(5,0.1)(6.5,0.1)
    \psline{-}(6.8,0)(6.8,-0.2) \uput[d](6.8,-0.1){\small
      $\theta_{11}$} \psline[linestyle=dashed, dash=1.5mm
    0.5mm]{-}(5.7,-0.1)(6.8,-0.1)
    \psline{-}(9,0)(9,0.2)
    \psline{-}(8.5,0)(8.5,-0.2) \uput[d](7.65,-0.1){$\dots$}
    \uput[d](8.5, -0.1){\small $\theta_{1J_1}$}
    \psline{-}(10,-0.2)(10,0.2) \uput[u](10,0.2){\small $T$}
    \psline[linestyle=dashed, dash=2mm 0.5mm]{-}(9,0.1)(10,0.1)
    \psline[linestyle=dashed, dash=1.5mm 0.5mm]{-}(8.5,-0.1)(10,-0.1)
  \end{pspicture}
  \caption{The different sequences of vector fields in the
    finite-dimensional sub-problem.}
\end{figure}
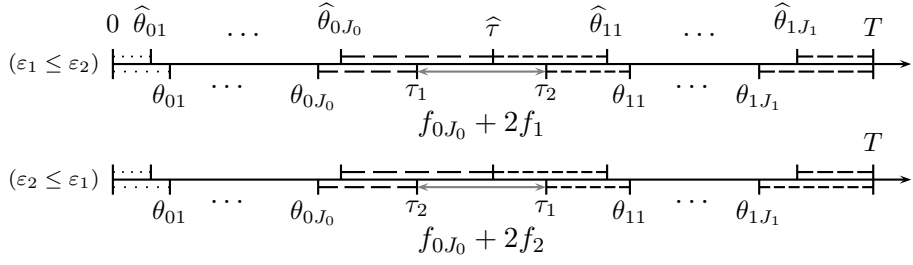

We shall denote the solution, evaluated at time $t$, of \eqref{eq:sub}
emanating from a point $x \in M$ at time $0$, as $S_t(x, \delta,
\e)$. Observe that $S_t(x,0,0) = \wh S_t(x)$.

Notice that the reference control is achieved along $\ep_1 = \ep_2$,
that is the reference flow is attained by (FP) on a point of
non-differentiability of the functions
\[
\theta_{0,J_0 +1} := \hat\tau + \min\{ \ep_1, \ \ep_2 \}, \qquad
\theta_{10} := \hat\tau + \max\{ \ep_1, \ \ep_2 \} .
\]
We are going to prove (see Remark \ref{rem:L} in Section
\ref{sec:lambda}) that despite this lack of differentiability of the
switching times $\theta_{0J_0}$, $\theta_{10}$, (FP) is $C^1$ (indeed
$C^{1,1}$) at $\delta_{ij} = \ep_1 = \ep_2 = 0$

We can thus consider, on the kernel of the first variation of (FP),
its second variation, piece-wisely defined as the second variation of
the restrictions of (FP) to the half-spaces $\{(\delta, \ep) \colon
\ep_1 \leq \ep_2 \}$ and $\{(\delta, \ep) \colon \ep_2 \leq \ep_1
\}$. Because of the structure of (FP), this second variation is
coercive if and only if both restrictions are positive-definite
quadratic forms. In particular any of their convex combinations is
positive-definite on the kernel of the first variation, i.e.~Clarke's
generalized Hessian at $(x, \delta, \ep) = (\wh x_0, 0,0)$ is
positive-definite on that kernel, see Remark \ref{rem:2var} in Section
\ref{sec:lambda}. 

In Section \ref{sec:lambda} we give explicit formulas both for the
first and for the second variations.  We shall ask for such second
variations to be positive definite and prove the following theorem:
\begin{theorem}
  \label{thm:main}
  Let $(\wh\xi, \wh u)$ be a bang-bang regular extremal (in the sense
  of Assumption \ref{ass:1}) for problem \eqref{eq:minipb} with
  associated covector $\wh\lambda$.  Assume all the switching times of
  $(\wh\xi , \wh u)$ but one are simple, while the only non-simple
  switching time is double.

  Assume the strong Legendre conditions, Assumptions \ref{ass:2} and
  \ref{ass:3}, hold.  Assume also that the second variation of problem
  (FP) is positive definite on the kernel of the first variation. Then
  $(\wh\xi, \wh u )$ is a strict strong local optimizer for problem
  \eqref{eq:minipb}. If the extremal is abnormal ($p_0 = 0$), then
  $\wh\xi$ is an isolated admissible trajectory.
\end{theorem}

\section{Hamiltonian methods}\label{sec:hamimeth}

The proof will be carried out by means of Hamiltonian methods, which
allow us to reduce the problem to a finite dimensional one defined in
a neighborhood of the final point of the reference trajectory. For a
general introduction to such methods see e.g.~\cite{AS04}. We repeat
here the argument for the sake of completeness.

In Section \ref{sec:maxflow} we prove that the maximized Hamiltonian
of the control system, $H$, is
well defined and Lipschitz continuous on the whole cotangent bundle
$T^*M$. Its Hamiltonian vector field $\vH$ is piecewise smooth in a
neighborhood of the range of $\wh\lambda$ and its flow, which we denote
as
\[
\cH \colon (t, \ell) \in [0, T ] \times T^*M \mapsto \cH_t(\ell) \in
T^*M,
\]
is well defined in a neighborhood of $[0, T] \times
\{\lo\}$  and $\wh\lambda$
is a trajectory of $\vH$: $\dfrac{\ud}{\ud t}{\wh\lambda}(t)=
\vH(\wh\lambda(t))$, i.e.~$\hat\lambda(t) = \cH_t(\lo)$.

In Sections \ref{sec:lambda}-\ref{sec:invert} we prove that there exist a $C^2$
function $\a$ such that $\left. \alpha \right\vert_{N_0} = p_0 c_0$,  
$\ud\alpha(x_0) = \lo $ and enjoying the following
property: the map 
\[ 
\id \times \pi \cH \colon (t, \ell) \in [0, T ]\times \Lambda
\mapsto (t, \pi\cH_t(\ell)) \in [0, T ] \times M
\]
is one--to--one onto a neighborhood of the graph of $\wh\xi$, where 
$
\Lambda := \left\{ \ud\a(x) \colon x \in \cO(x_0) \right\}$.
Indeed the proof of this invertibility is the main core of the paper
and its main novelty.

Under the above conditions the one--form $ \omega :=
  \cH^* ( p \ud q - H \ud t ) $ is exact on $[0, T] \times \Lambda$,
  hence there exists a $C^1$ function
\[
\chi \colon (t, \ell) \in [0, T] \times \Lambda \mapsto
\chi_t(\ell) \in \R
\]
such that $\ud\chi = \omega$. Also it may be shown (see,
e.g.~\cite{ASZ02b}) that $\ud \, (\chi_t \circ (\pi\cH_t)^{-1}) =
\cH_t \circ (\pi\cH_t)^{-1}$ for any $t \in [0, T]$.  Moreover we may
assume $\chi_0 = \alpha \circ \pi$

Observe that $(t, \wh\xi(t)) = (\id \times \pi \cH ) (t, \lo)$ and let
us show how this construction leads to the reduction. Define
\[
\cV := ( \id \times \pi \cH )( [0, T ]\times \Lambda) , \qquad \psi :=
(\id \times \pi\cH)^{-1} \colon \cV \to [0, T ]\times \Lambda
\]
and let $(\xi, u)$ be an admissible pair (i.e.~a pair satisfying
\eqref{eq:2}--\eqref{eq:3}--\eqref{eq:4}) such that the graph of $\xi$
is in $\cV$. We can obtain a closed path $\Gamma$ in $\cV$ with a
concatenation of the following paths:
\begin{itemize}
\item $\Xi \colon t \in [0, T] \mapsto (t, \xi(t)) \in \cV$,
\item $\Phi_T \colon s \in [0, 1] \mapsto (T, \phi_T(s)) \in \cV$,
  where $\phi_T \colon s \in [0, 1] \mapsto \phi_T(s) \in M$ is such
  that $\phi_T(0) = \xi(T)$, $\phi_T(1) = \wh x_f$,
\item $\wh\Xi \colon t \in [0, T] \mapsto (t,\wh \xi(t)) \in \cV$, ran
  backward in time,
\item $\Phi_0 \colon s \in [0, 1] \mapsto (0, \phi_0(s)) \in \cV$,
  where $\phi_0 \colon s \in [0, 1] \mapsto \phi_0(s) \in M$ is such
  that $\phi_0(0) = \wh x_0$, $\phi_0(1) = \xi(0)$.
\end{itemize}
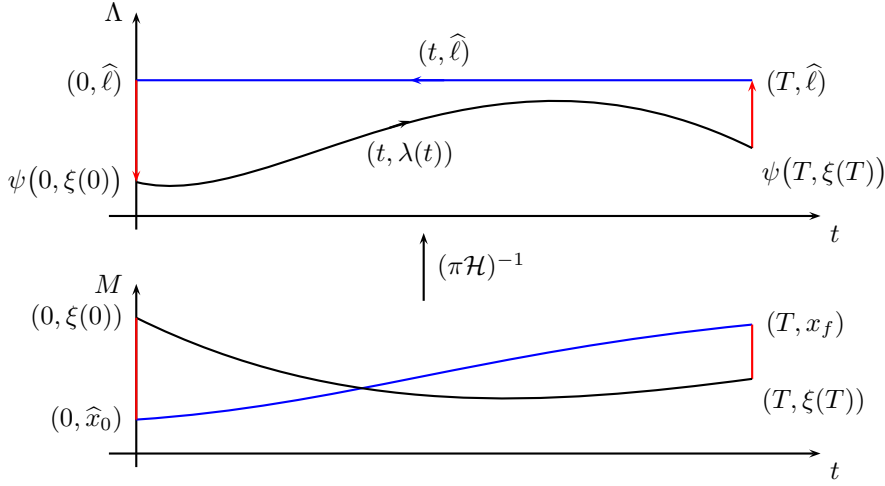
\begin{figure}[h!]
  \centering
 \begin{pspicture}(-1,4.5)(10,-1.8) \psset{unit=0.9cm}
      \psline{->}(0,1.8)(0,5) \psline{->}(-0.4,2)(10,2)
      \uput[l](0,5){\small$\Lambda$} \uput[dr](10,2){\small$t$}
 
      \uput[l](0,4){\small$(0, \lref)$} 
      \uput[l](0,2.5){\small$\psi\big(0,\xi(0)\big)$}

      \psline[linecolor=blue]{-}(0, 4)(9, 4)
      \psline[linecolor=blue]{<-}(4, 4)(4.5, 4)
 
      \uput[r](9, 4){\small$(T, \lref)$}
      \uput[dr](9, 3){\small$\psi\big(T,\xi(T)\big)$}
      \uput[u](4.5, 4){\small$(t, \lref)$}
     
      \psbezier[linecolor=red]{<-}(9,4)(9, 3.5)(9, 3.5)(9, 3)
 
      \psbezier[linecolor=black]{-}(9,3)(5,5)(2,2)(0,2.5)
      \psline{<-}(4, 3.4)(3.7, 3.3)
      \uput[d](4, 3.3){\small$(t,  \lambda(t))$}
      \psline[linecolor=red]{<-}(0, 2.5)(0, 4)
 
      \psline{<-}(4.2,1.75)(4.2,0.75) \uput[r](4.2,1.25){\small
        $(\pi\cH)^{-1}$}

      \psline{->}(0,-1.7)(0,1) \psline{->}(-0.4,-1.5)(10,-1.5)
      \uput[l](0,1){\small$M$} \uput[dr](10,-1.5){\small$t$}

      \uput[l](0,-1){\small$(0,\wh x_0)$} \uput[l](0,0.5){\small$(
        0,\xi(0))$}

      \psbezier[linecolor=blue]{-}(0,-1)(3,-0.8)(5,0)(9,0.4)
      \psline[linecolor=red]{-}(9,-0.4)(9,0.4)
      \psbezier[linecolor=black]{-}(0,0.5)(3,-1)(6,-0.8)(9,-0.4)
      \psline[linecolor=red]{-}(0,0.5)(0,-1)
      \uput[r](9,0.4){\small$(T,x_f)$}
      \uput[dr](9,-0.4){\small$( T,\xi(T))$}
    \end{pspicture}
  \caption{The closed path $\Gamma$ and its preimage}
  \label{fig:circ}
\end{figure}
Since the one-form $\omega$ is exact we get
\begin{equation*}
  0 = \oint_\Gamma \omega = 
  \int_{\psi(\Xi)} \omega
  + \int_{\psi(\Phi_T)}\omega 
  - \int_{\psi(\wh\Xi)} \omega   + \int_{\psi(\Phi_0)}\omega .
\end{equation*}
From the definition of $\omega$ and the maximality properties of $H$
we get
\begin{equation}
  \int_{\psi(\wh\Xi)} \omega
  =  0 ,  \qquad
  \int_{\psi(\Xi)}  \omega
  \leq  0 \label{eq:integ}
\end{equation}
so that
\begin{equation} \label{eq:ineq} \int_{\psi(\Phi_T)}\omega +
  \int_{\psi(\Phi_0)}\omega \geq 0.
\end{equation} 
Since
\begin{gather*}
  \int_{\psi(\Phi_T)}\omega = \int_{(\pi\cH_T)^{-1}\circ\Phi_T} \hspace{-5mm}
  \ud\, (\chi_T\circ(\pi\cH_T)^{-1}) =
  \chi_T\circ(\pi\cH_T)^{-1}(\wh x_f) -
  \chi_T\circ(\pi\cH_T)^{-1}(\xi(T)) , \\
  \int_{\psi(\Phi_0)}\omega = \int_0^1
  \scal{\ud\a(\phi_0(s))}{\dot\phi_0(s)} \ud s = \alpha(\xi(0))  -
  \alpha(\wh x_0) ,
\end{gather*} 
inequality \eqref{eq:ineq} yields
\begin{equation}\label{eq:disequa}
  \alpha(\xi(0)) - \alpha(\wh x_0) + \chi_T\circ (\pi\cH_T)^{-1}(\wh x_f) - 
  \chi_T \circ (\pi\cH_T)^{-1}(\xi(T)) \geq 0.
\end{equation}
Thus
\begin{multline}
  \label{eq:riduci}
  \alpha(\xi(0)) + \beta(\xi(T))
  - \alpha(\wh x_0) - \beta(\wh x_f) \geq \\
  \geq \left(\chi_T \circ (\pi\cH_T)^{-1} + \beta \right)(\xi(T)) -
  \left( \chi_T \circ (\pi\cH_T)^{-1} + \beta \right)(\wh x_f)
\end{multline}
that is: we only have to prove the local minimality at $\wh x_f$ of the
function
\[
\cF \colon x \in N_f \cap \cO(\wh x_f) \mapsto \left( \chi_T \circ
  (\pi\cH_T)^{-1} + \beta \right)(x) \in \R .
\]
where $\cO(\wh x_f)$ is a small enough neighborhood of $\wh x_f$.

In proving both the invertibility of $\id \times \pi\cH$ and the local
minimality of $\wh x_f$ for $\cF$ we shall exploit the positivity of
the second variations of problem (FP).  See \cite{AG90, AG97, AS04}
for a more general introduction to Hamiltonian methods.

\section{The maximized flow}\label{sec:maxflow}
We are now going to prove the properties of the maximized Hamiltonian
$H$ and of the flow - given by classical solutions - of the
associated Hamiltonian vector field $\vH$. 
Such flow will turn out to be Lipschitz continuous and
piecewise-$C^1$.  In such construction we shall use 
only the regularity assumptions \ref{ass:1}-\ref{ass:2}-\ref{ass:3}
and not the positivity of the second variations of problems (FP).

We shall proceed as follows:
\begin{itemize}
\item[\em Step 1:] we first consider the simple switches occurring
  before the double one. We shall explain the procedure in details for
  the first simple switch. The others are treated iterating such
  procedure \cite{ASZ02b};
\item[\em Step 2:] we decouple the double switch obtaining two simple
  switches that might coincide and that give rise to as many flows;
\item[\em Step 3:] We consider the simple switches that occur after
  the double one. For each of the flows originating from the double
  switch we apply the same procedure of Step 1.
\end{itemize}
\begin{itemize}
\item[\em Step 1:] Regularity Assumption \ref{ass:1} implies that
  locally around $\lo$, the maximized Ha\-mil\-to\-nian is $K_{00}$
  and that $\wh\lambda(t)$, i.e.~the flow of $\vK_{00}$ evaluated in
  $\lo$, intersects the level set $\{\ell \in T^*M \colon K_{01}(\ell)
  = K_{00}(\ell)\}$ at time $\hat\theta_{01}$. Assumption \ref{ass:2}
  yields that such intersection is transverse. This suggests us to
  define the switching function $\theta_{01}(\ell)$ as the time when
  the flow of $\vK_{00}$, emanating from $\ell$, intersects such level
  set and to switch to the flow of $\vK_{01}$ afterwards.  To be more
  precise, we apply the implicit function theorem to the map
  \[
  \Phi_{01}(t, \ell) := (K_{01} - K_{00})\circ\exp t\vK_{00}(\ell)
  \]
  in a neighborhood of $(t, \ell) := (\hat\theta_{01}, \lo)$ in $[0,
  T] \times T^*M$, so that $H(\ell) = K_{00}(\ell)$ for
  any $t \in [0, \theta_{01}(\ell)]$.  We then iterate this procedure
  and obtain the switching surfaces $\{ (\theta_{0j}(\ell), \ell)
  \colon \ell \in \cO(\lo)\}$, $j=1, \ldots, J_0$ where:
  \[
  \theta_{00}(\ell) := 0 \qquad\phi_{00}(\ell) :=\ell
  \]
  and, for $ j=1, \ldots, J_0$, we have
  \begin{itemize}
  \item $\theta_{0j}(\ell)$ is the unique solution to
    \[
    \left( K_{0j} - K_{0, j-1} \right)\circ\exp \theta_{0j}(\ell)
    \vK_{0, j-1} \left(\phi_{0, j-1}(\ell) \right) =0
    \]
    defined by the implicit function theorem in a neighborhood of
    $(t, \ell) = (\hat\theta_{0j}, \lo)$;
  \item $\phi_{0 j}(\ell)$ is defined by
    \begin{equation}
      \phi_{0 j}(\ell) := \exp \big( -\theta_{0j}(\ell) \vK_{0j}
      \big) \circ \exp \theta_{0j}(\ell)\vK_{0, j-1} \left(\phi_{0,
          j-1}(\ell) \right). \label{eq:phi0j}
    \end{equation}
  \end{itemize}

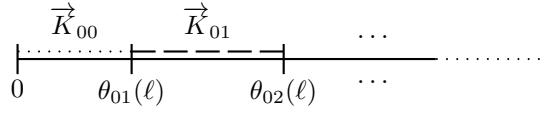
\begin{figure}[ht!]
\centering
      \begin{pspicture}(0,-0.5)(7,0.5)
        \psline{-}(0,0)(5.5,0)
        \psline[linestyle=dotted]{-}(5.5,0)(7,0)
        \psline{-}(0,-.2)(0,.2) \uput[d](0,-0.1){\small $0$}
        \psline{-}(1.5,-0.2)(1.5,0.2) 
        \uput[d](1.5,-0.1 ){\small $\theta_{01}(\ell)$}
        \psline[linestyle=dotted]{-}(0,0.1)(1.5,0.1)
        \uput[u](.75,0.1){\small $\vK_{00}$}
        \psline{-}(3.5,-0.2)(3.5,0.2)
        \uput[u](2.5,0.1){\small $\vK_{01}$}
        \psline[linestyle=dashed, dash=3mm 1mm]{-}(1.5,0.1)(3.5,0.1)
        \uput[d](3.5,-0.1 ){\small $\theta_{02}(\ell)$}
        \uput[d](4.7,-0.1 ){\small \ldots}
        \uput[u](4.7,0.1 ){\small \ldots}
      \end{pspicture}
\caption{Construction of the maximized flow.}
\end{figure}

\item[\em Step 2: ] Let us now show how to decouple the double
  switching time in order to define the maximized Hamiltonian
  $H(\ell)$ in a neighborhood of $(\hat\tau,
  \wh\lambda(\hat\tau))$.
  In this we depart from \cite{ASZ02b} in that we introduce the new
  vector fields $k_1$, $k_2$ in the sequence of values assumed by the
  reference vector field. We do this in five stages:
  \begin{itemize}
  \item for $\nu =1, \ 2$ let $\tau_\nu(\ell) $ be the unique solution
    to
    \[
    2F_\nu \circ \exp \tau_\nu
    (\ell)\vK_{0J_0} (\phi_{0 J_0}(\ell)) =
    \left( K_\nu - K_{0 J_0} \right)\circ \exp \tau_\nu
    (\ell)\vK_{0J_0} (\phi_{0 J_0}(\ell)) = 0
    \]
    defined by the implicit function theorem in a neighborhood of
    $(\hat\tau, \lo)$;
  \item choose
    \begin{equation*}
      \theta_{0, J_0+1}(\ell) := \min \left\{ \tau_1(\ell),
        \tau_2(\ell)
      \right\};
    \end{equation*}
  \item for $\nu =1, 2$, define
    \begin{equation*}
      \phi_{0, J_0 + 1}^\nu(\ell) := \exp \big( - \tau_\nu(\ell)\vK_\nu \big)\circ \exp \tau_\nu(\ell) \vK_{0 J_0}
      \left(\phi_{0 J_0}(\ell)\right),
    \end{equation*}
    and let $\theta_{10}^\nu(\ell)$ be the unique solution to
    \begin{multline*}
      \quad 2F_{3-\nu}\circ\exp \theta_{10}(\ell) \vK_\nu
    \left(\phi_{0, J_0+1}^\nu(\ell) \right) = \\
    = \left( K_{10} - K_\nu \right)\circ\exp \theta_{10}(\ell) \vK_\nu
    \left(\phi_{0, J_0+1}^\nu(\ell) \right) =0 \quad 
    \end{multline*}
    defined by the implicit function theorem in a neighborhood of
    $(\hat\tau, \lo)$;
  \item for $\nu=1, 2$ define
    \[
    \phi_{10}^\nu := \exp \big( - \theta_{10}^\nu(\ell)\vK_{10} \big)
    \circ \exp \theta_{10}^\nu(\ell) \vK_\nu \left( \phi^\nu_{0, J_0 +
        1}(\ell) \right);
    \]
  \item choose
    \[
    \theta_{10} (\ell) =
    \begin{cases}
      \theta^1_{10}(\ell) \ & \text{if } \tau_1(\ell) \leq \tau_2(\ell), \\
      \theta^2_{10}(\ell) \ & \text{if } \tau_2(\ell) < \tau_1(\ell).
    \end{cases}
      \]
  \end{itemize}
  Notice that if $\tau_1(\ell) = \tau_2(\ell)$, then
  $\theta^1_{10}(\ell) = \theta^2_{10}(\ell) = \tau_1(\ell) =
  \tau_2(\ell)$ so that $\theta_{10}(\cdot)$ is continuous. To be more
  precise, the function $\theta_{10}(\cdot)$ is Lipschitz continuous
  on its domain and is actually $C^1$ on its domain but with the only
  possible exception of the set $\{\ell \in T^*M \colon \tau_1(\ell) =
  \tau_2(\ell)\}$. 
\item[\em Step 3: ] Finally we define analogous quantities for the
  simple switching times that follow the double one. For each $j=1,
  \ldots, J_1$ we proceed in three stages:
  \begin{itemize}
  \item for $\nu=1, 2$ let $\theta^\nu_{1j}(\ell)$ be the unique
    solution to
    \[
    \left( K_{1j} - K_{1, j-1} \right)\circ\exp
    \theta^\nu_{1j}(\ell)\vK_{1,j-1} \left(\phi^\nu_{1, j-1}(\ell)
    \right) = 0
    \] defined by the implicit function theorem in a neighborhood of
    $(\hat\theta^\nu_{1j}, \lo)$;
  \item define
    \[
    \phi^\nu_{1 j}(\ell) := \exp \big( - \theta^\nu_{1j}(\ell)
    \vK_{1j} \big)\circ \exp \theta^\nu_{1j}(\ell) \vK_{1, j-1}
    \left(\phi^\nu_{i, j-1}(\ell) \right);
    \]
  \item choose
    \[
    \theta_{1j}(\ell) = \begin{cases}
      \theta^1_{1j}(\ell) \ & \text{if } \tau_1(\ell) \leq \tau_2(\ell) \\
      \theta^2_{1j}(\ell) \ & \text{if } \tau_2(\ell) < \tau_1(\ell).
    \end{cases}
    \]
  \end{itemize}
  We conclude the procedure defining $\theta_{1, J_1+1}(\ell) =
  \theta^1_{1, J_1+1}(\ell) = \theta^2_{1, J_1+1}(\ell):= T$.
\end{itemize}
To justify the previous procedure we have to show that we can actually
apply the implicit function theorem to define the switching times
$\theta_{ij}(\ell)$ and that they are ordered as follows:
\[
\theta_{0, j-1}(\ell) < \theta_{0j}(\ell) \ldots <\theta_{0J_0}(\ell)
< \theta_{0, J_0 + 1} (\ell) \leq \theta_{10}(\ell) <
\theta_{11}(\ell) < \ldots .
\]
We prove it with an induction argument. The functions
$\theta_{00}(\cdot)$ and $\phi_{00}(\cdot)$ are obviously well
defined. Assume that $\theta_{0j}$, $\phi_{0j} $ are well defined for
some $j \geq 1$ and let
\begin{equation*}
  \Phi_{0, j+1}(t,\ell) = \left( K_{0, j+1} - K_{0,j}\right) 
  \circ \exp t \vK_{0j}
  \circ\phi_{0j}(\ell).
\end{equation*}
Then one can compute
\begin{equation*}
  \left. \dfrac{\partial\Phi_{0, j+1}}{\partial t}
  \right\vert_{(\hat\theta_{0, j+1}, \lo)} = \dueforma{\vK_{0j}}{\vK_{0, j+1}} (\wh\lambda(\hat\theta_{0, j+1})) 
\end{equation*}
which is positive by Assumption \ref{ass:2}, so that the implicit
function theorem yields the $C^1$ function $\theta_{0, j+1}$.  Thus,
we also get a $C^1$ function $\phi_{0, j+1}$ by equation
\eqref{eq:phi0j}.  By induction, the $\theta_{0j}$'s are well defined
for any $j= 1, \ldots, J_0$ and, by continuity, the order is preserved
for $\ell$ in a neighborhood of $\lo$.  Also, the implicit function
theorem yields a recursive formula for the linearizations of
$\theta_{0j}$ and $\varphi_{0j}$ at $\lo$:
\begin{equation}\label{eq:lintheta0j}
  \scal{\ud\theta_{0j}(\lo)}{\dl} = \dfrac{- \dueforma{\exp(\hat\theta_{0j}\vK_{0, j-1})_*\varphi_{0, j-1 \, *}(\dl)}{ (\vK_{0j} - \vK_{0,j-1}) (\wh\lambda(\hat\theta_{0j}))}}{\dueforma{\vK_{0, j-1}}{\vK_{0j}}(\wh\lambda(\hat\theta_{0j}))}
\end{equation}
\begin{multline}\label{eq:linphi0j}
  \varphi_{0j*}(\dl) = \exp(- \hat\theta_{0j}\vK_{0j})_* \Big\{
  - \scal{\ud\theta_{0j}(\lo)}{\dl}(\vK_{0j} - \vK_{0, j-1})(\wh\lambda(\hat\theta_{0j})) + \\
  + \exp(\hat\theta_{0j}\vK_{0, j-1})_* \varphi_{0, j-1 \, *}(\dl)
  \Big\}.
\end{multline}
Let us show that $\theta_{0, J_0 +1}$ and $\theta_{10}$ are also well
defined. Let
\begin{equation*}
  \Psi_\nu(t,\ell) = \left(K_\nu - K_{0 J_0} \right) 
  \circ \exp t \vK_{0 J_0}  \circ\phi_{0J_0}(\ell)
  \qquad \nu =1, 2.
\end{equation*}
Then
\[ \left. \dfrac{\partial\Psi_\nu}{\partial t} \right\vert_{(\hat\tau
  , \lo)} = \dueforma{\vK_{0 J_0}}{\vK_\nu} (\wh\lambda(\hat\tau ))
\quad \nu = 1,2
\]
which are positive by Assumption \ref{ass:3}, so that $\tau_1(\cdot)$
and $\tau_2(\cdot)$ are both well defined again by means of the
implicit function theorem.

Now let
\begin{equation*}
  \Phi_{10}^\nu(t,\ell) = \left( K_{10} - K_\nu\right) 
  \circ \exp t \vK_\nu \circ\phi^\nu_{0, J_0 + 1}(\ell), \quad \nu=1,2
\end{equation*}
then
\[
\left. \dfrac{\partial\Phi_{1 0}^\nu}{\partial t}
\right\vert_{(\hat\tau, \lo)} = \dueforma{\vK_\nu}{\vK_{1 0}}
(\wh\lambda(\hat\tau)) , \quad \nu=1,2
\]
which are positive again by Assumption \ref{ass:3}, and the same
argument applies.

As already mentioned, by assumption $\hat\theta_{0, j-1} < \hat\theta_{0j}$ and
$\hat\theta_{0J_0} < \hat\tau$ so that, by continuity, $\theta_{0, j-1}(\ell)
< \theta_{0j}(\ell)$ and $\theta_{0J_0}(\ell) < \theta_{0,
  J_0+1}(\ell) = \min \{\tau_1(\ell), \tau_2(\ell)\}$ for any $\ell$
in a sufficiently small neighborhood of $\lo$.

Let us now show that $\theta_{0, J_0+1}(\ell) \leq \theta_{10}(\ell)$.
We examine all the possibilities for $\tau_1(\ell)$ and
$\tau_2(\ell)$:
\begin{itemize}
\item assume $\ell$ is such that $\theta_{0, J_0+1}(\ell) =
  \tau_1(\ell) < \tau_2(\ell)$. Since $\Psi_2(\tau_2(\ell), \ell) 
  = 0$ one has
  \begin{equation*}
    \begin{split}
    \Psi_2(t, \ell) & = \dfrac{\partial\Psi_2}{\partial t}(\tau_2(\ell),
    \ell) (t - \tau_2(\ell)) + o ( t - \tau_2(\ell) )  = \\
    & = (t - \tau_2(\ell)) \left(
      \left. \dueforma{\vK_{0J_0}}{\vK_2}\right\vert_{\exp\tau_2(\ell)\vK_{0J_0}\circ
        \phi_{0J_0}(\ell)} + o(1)\right). 
    \end{split}
  \end{equation*}
  In particular, choosing $t = \theta_{0, J_0 +1}(\ell) =
  \tau_1(\ell)$, by Assumption \ref{ass:3} and by continuity, when
  $\ell$ is sufficiently close to $\lo$, we have $
  \Upsilon_{\ell}(\theta_{0, J_0 +1}(\ell)) < 0$, that is:
  \begin{equation}\label{eq:segue}
    \Psi_2(\theta_{0, J_0 + 1}(\ell), \ell) = 
    \left( K_2 - K_{0J_0} \right)\circ\exp \theta_{0, J_0+1}(\ell)
    \vK_{0 J_0} \circ\phi_{0J_0}(\ell) < 0 . 
  \end{equation}
  Since $ K_2 - K_{0J_0} = 2F_2 = K_{10} - K_1$, equation \eqref{eq:segue}
  can also be written as
  \begin{equation*}
    0 > \left(
      K_{10} - K_1 \right) \circ \exp 0 \vK_1 
    \circ\exp \theta_{0, J_0+1}(\ell)\vK_{0J_0}
    \circ\phi_{0J_0}(\ell),
  \end{equation*}
  i.e.~the switch of the component $u_2$ has not yet occurred at time
  $\tau_1(\ell)$, so that $\theta_{10}^1(\ell) - \tau_1(\ell) >0 $.
\item Analogous proof holds if $\theta_{0, J_0+1}(\ell) = \tau_2(\ell)
  < \tau_1(\ell)$,
\item If $\ell$ is such that $\tau_1(\ell) = \tau_2(\ell)$, then
  $\theta_{10}(\ell) = \theta_{0, J_0+1}(\ell)$.
\end{itemize}
For the simple switches occurring after the double one, by continuity,
we have:
\[
\theta_{1j}(\ell) \leq \max \{\theta_{1j}^1(\ell), \theta_{1j}^2(\ell)
\} < \min \{\theta_{1,j+1}^1(\ell), \theta_{1,j+1}^2(\ell) \} \leq
\theta_{1, j+1}(\ell)
\]
for $\ell$ in a sufficiently small neighborhood of $\lo$.

For the purpose of future reference we report here the expression for
the differentials of the $\theta_{0j}$'s, $\tau_\nu$'s and
$\theta_{1j}^\nu$'s, and of the $\phi_{0j \, *}$'s $\phi^\nu_*$'s and
$\phi^\nu_{1j \, *}$'s. Such formulas can be proved with an induction
argument.
\begin{lemma}
  For any $j=1, \ldots, J_0$ consider the following endomorphism of $
  T_{\lo}(T^*M)$:
  \begin{equation}
    \label{eq:delta0j}
    \Delta_{0j} \dl =  \dl - \dsum_{s=1}^{j}
    \scal{\ud\theta_{0s}(\lo)}{\dl}(\vG_{0s} - \vG_{0, s-1})(\lo).
  \end{equation}
  Then
  \begin{align}\label{eq:lintheta0jind}
    \scal{\ud\theta_{0j}(\lo)}{\dl} & = \dfrac{ - \dueforma{
        \Delta_{0,j-1} \dl }{ (\vG_{0j} - \vG_{0,j-1})
        (\lo)}}{\dueforma{\vG_{0,
          j-1}}{\vG_{0j}}(\lo)},  \\
    \label{eq:linphi0jind}
    \varphi_{0j*}(\dl) & = \exp(- \hat\theta_{0j}\vK_{0j})_*
    \wh\cH_{\hat\theta_{0j\, *}}\Delta_{0j} \dl ,\\
    \label{eq:lintaunuind}
    \scal{\ud\tau_\nu(\lo)}{\dl} & = \dfrac{ - \dueforma{
        \Delta_{0J_0} \dl }{ (\vH_{\nu} - \vG_{0J_0})
        (\lo)}}{\dueforma{\vG_{0 J_0}}{\vH_{\nu}}(\lo)} , \\
    \label{eq:lintheta10ind}
    \begin{split}
      \scal{\ud\theta^\nu_{10}(\lo)}{\dl} & =
      \dfrac{-1}{\dueforma{\vH_\nu}{\vG_{10}}(\lo)} \\
      & \bsi \Big(\Delta_{0J_0} \dl -
      \scal{\ud\tau_\nu(\lo)}{\dl}(\vH_{\nu} - \vG_{0J_0}) (\lo)\, ,
      \, (\vG_{10} - \vH_\nu)(\lo) \Big)
    \end{split} \\
    \intertext{and}
    \label{eq:linphinuind}
    \varphi^\nu_{0, J_0+1 \, *}(\dl) & = \exp(- \hat\tau\vK_{\nu})_*
    \wh\cH_{\hat\tau_*} \Big(\Delta_{0J_0}\dl -
    \scal{\ud\tau_\nu(\lo)}{\dl}(\vH_\nu - \vG_{0J_0})(\lo) \Big).
  \end{align}
  Moreover
  \begin{equation}\label{eq:tautheta}
    \begin{split}
      & \scal{\ud\theta^1_{10}(\lo)}{\dl} = \scal{\ud\tau_1(\lo)}{\dl}
      - \scal{\ud \, (\tau_1 - \tau_2)(\lo)}{\dl}
      \frac{\dueforma{\vG_{0J_0}}{\vH_2}(\lo)}{\dueforma{\vH_1}{\vG_{10}}(\lo)}, \\
      & \scal{\ud\theta^2_{10}(\lo)}{\dl} = \scal{\ud\tau_2(\lo)}{\dl}
      - \scal{ \ud \, (\tau_2 - \tau_1)(\lo)}{\dl}
      \frac{\dueforma{\vG_{0J_0}}{\vH_1}(\lo)}{\dueforma{\vH_2}{\vG_{10}}(\lo)}.
    \end{split}
  \end{equation}
  Also, for $\nu =1, 2$ and $j=0, \ldots, J_1$ consider the
  endomorphisms
  \begin{equation}
    \label{eq:delta1jnu}
    \begin{split}
      & \Delta^\nu_{1j} \dl = \Delta_{0J_0}\dl
      - \scal{\ud\tau_\nu(\lo)}{\dl}(\vH_\nu - \vG_{0J_0})(\lo) - \\
      & - \scal{\ud\theta^\nu_{10}(\lo)}{\dl}(\vG_{10} - \vH_\nu)(\lo)
      - \dsum_{s=1}^j \scal{\ud\theta_{1s}^\nu(\lo)}{\dl}\left(
        \vG_{1s} - \vG_{1, s-1} \right)(\lo)
    \end{split}
  \end{equation}
  Then
  \begin{align}
    \label{eq:linphi10ind}
    \varphi^\nu_{10 \, *}(\dl) & = \exp(- \hat\theta_{10}\vK_{10})_*
    \wh\cH_{\hat\theta_{10 \, *}} \Delta^\nu_{10}\dl, \\
    \label{eq:lintheta1jind}
    \scal{\ud\theta^\nu_{1j}(\lo)}{\dl} & = \dfrac{- \, \dueforma{
        \Delta^\nu_{1, j-1} \dl }{ (\vG_{1j} - \vG_{1,
          j-1})(\lo)}}{\dueforma{\vG_{1, j-1}}{\vG_{1j}}(\lo)} , \\
    \intertext{and}
    \label{eq:linphi1jind}
    \varphi^\nu_{1j \, *}(\dl) & = \exp(- \hat\theta_{1j}\vK_{1j})_*
    \wh\cH_{\hat\theta_{1j \, *}} \Delta^\nu_{1j} \dl .
  \end{align}
\end{lemma}
Thus we get that the flow of the maximized Hamiltonian coincides with
the flow of the Hamiltonian $H\colon (t,\ell) \in [0, T] \times T^*M
\mapsto H_t(\ell) \in T^*M $:
\begin{equation}
  \label{eq:maxfl}
  H\colon (t,\ell) \in [0, T] \times T^*M \mapsto H(\ell) \in T^*M 
\end{equation}
\begin{equation*}
  H_t(\ell) :=
  \begin{cases}
    K_{0j}(\ell) &  t \in (\theta_{0j}(\ell), \theta_{0, j+1}(\ell)], \quad j=0, \ldots, J_0  \\
    K_\nu(\ell) & t \in (\theta_{0, J_0 +1}(\ell), \theta_{10}(\ell)]
    , \quad
    \theta_{0, J_0 + 1}(\ell) = \tau_\nu (\ell)\\
    K_{1j}(\ell) & t \in (\theta_{1j}(\ell), \theta_{1, j+1}(\ell)],
    \quad j=0, \ldots, J_1.
  \end{cases}
\end{equation*}
\section{The second variation}
\label{sec:lambda}
To choose an appropriate horizontal Lagrangian manifold $\Lambda$ we
need to write the second variations of sub-problem (FP) and exploit
their positivity.  To write an invariant second variation, as
introduced in \cite{ASZ98a}, we write the pull-back $\zeta_t(x,
\delta, \e)$ of the flows $S_t$ along the reference flow $\wh S_t$,
which also permits us to analyze the influence of the double switch on
the final point of trajectories.

For the sake of greater clarity we first clear the field of all the
notational difficulties by performing our analysis in the case when
only the double switch occurs. Only afterwards we will discuss the
general case.

Let $\delta_{0 , J_0 + 1} := \min\{\ep_1, \ep_2 \}$, $\delta_{10} :=
\max\{\ep_1, \ep_2 \}$.  At time $ t = T$ we have
\begin{multline*}
  \zeta_T(x, \delta, \e) = \wh S_T^{-1}\circ S_T ( x, \delta, \e)
  = \exp\left(- \delta_{1 0}\right) g_{1 0} \circ \\
  \circ \exp\left(\delta_{10} - \delta_{0 1} \right) h_\nu \circ
  \exp\left( \delta_{01} - \delta_{00}\right) g_{0 J_0} (x)
\end{multline*}
where $\nu =1$ if $\ep_1 \leq \ep_2$, $\nu =2$ otherwise.
Let $\wt f_1$ and $\wt f_2$ be the pull--backs of $f_1$ and $f_2$ from
time $\hat\tau$ to time $t=0$, i.e., 
\[
\wt f_\nu := \wh S_{\hat\tau\, *}^{-1} f_\nu \circ \wh S_{\hat\tau},
\quad \nu=1, 2
\]
so that 
\[
h_\nu = g_{0J_0} + 2 \wt f_\nu, \quad \nu =1, 2 , \qquad 
g_{10} = g_{0J_0} + 2 \wt f_1 + 2 \wt f_2.
\]
The linearized flow at time $T$ has the following form:
\begin{equation*}
 L(\dx,\delta,\e)=\dx + (\delta_{11}-\delta_{01})g_{01}(x)+2(\delta_{11}-\e_1)
\wt f_1(x)+2(\delta_{11}-\e_2) \wt f_2(x),
\end{equation*}
which shows that the flow is $C^1$. 

Let us now go back to the general case: at time $ t = T$ we have
\begin{equation*}
  \begin{split}
    & \zeta_T(x, \delta, \e) = \wh S_T^{-1}\circ S_T ( x, \delta, \e)
    = \exp\left(- \delta_{1 J_1}\right) g_{1 J_1} \circ \ldots \circ
    \exp\left(\delta_{1 1} - \delta_{10} \right) g_{10}
    \circ \\
    & \circ \exp\left(\delta_{10} - \delta_{0, J_0 + 1} \right) h_\nu
    \circ \exp\left( \delta_{0, J_0 + 1} - \delta_{0 J_0}\right) g_{0
      J_0} \circ
    \ldots 
    \circ\exp\delta_{01}g_{00}(x)
  \end{split}
\end{equation*}
where $\nu =1$ if $\ep_1 \leq \ep_2$, $\nu =2$ otherwise.

Define
\begin{align*}
  a_{00} & := \delta_{01}; \displaybreak[0] \\
  a_{0j} & := \delta_{0, j+1} - \delta_{0 j}  \qquad  j = 1, \ldots, J_0 ; \displaybreak[0]\\
  b & := \delta_{10} - \delta_{0, J_0 +1 } ; \\
  a_{1j} & := \delta_{1, j+1} - \delta_{1 j}  \qquad  j = 0, \ldots, J_1 - 1; \displaybreak[0]\\
  a_{1 J_1} & := - \delta_{1 J_1}.
\end{align*}
Then $ b + \dsum_{i=0}^1\dsum_{j=0}^{J_i} a_{ij} = 0$ and, with a
slight abuse of notation, we may write
\begin{equation*}
  \begin{split}
    & \zeta_T(x, a, b)= \exp a_{1 J_1} g_{1 J_1} \circ \ldots \circ
    \exp a_{11} g_{11}
    \circ \exp a_{10} g_{10} \\
    & \circ \exp b h_\nu \circ \exp a_{0 J_0} g_{0 J_0} \circ \ldots
    \circ\exp a_{0 1} g_{01} \circ\exp a_{0 0} g_{0 0}(x),
  \end{split}
\end{equation*}-
where $\nu =1$ if $\ep_1 \leq \ep_2$, $\nu =2$ otherwise. Henceforward
we will denote by $a$ the $(J_0 + J_1 + 2)$-tuple $(a_{00}, \ldots,
a_{0J_0}, a_{10}, \ldots, a_{1J_1})$.

The reference flow is the one associated to $(a, b) = (0, 0)$.  The
first order approximation of $\zeta_T$ at a point $(x, 0, 0)$ is given
by
\begin{equation*}
  \begin{split}
    L(\dx, a, b) & = \dx + b h_\nu(x) +
    \sum_{i=0}^1\sum_{j=0}^{J_i}a_{ij}g_{ij}(x) = \\[-3mm]
    = & \, \dx + \sum_{j=0}^{J_0 - 1}a_{0j}g_{0j}(x)
    + (\delta_{0, J_0 + 1} - \delta_{0J_0 }) g_{0J_0}(x) + \\[-3mm]
    & + (\delta_{10} - \delta_{0, J_0 +1 })h_\nu(x) + (\delta_{11} -
    \delta_{10})g_{10}(x) + \sum_{j=1}^{J_1}a_{1j}g_{1j}(x)
  \end{split}
\end{equation*}
where $\nu =1$ if $\ep_1 \leq \ep_2$, $\nu =2$ otherwise.  Introduce
the pull-backs of $f_1$ and $f_2$ from time $\hat\tau$ to time $t=0$:
\[
\wt f_\nu := \wh S_{\hat\tau\, *}^{-1} f_\nu \circ \wh S_{\hat\tau}
\quad \nu=1, 2.
\]
Then $h_\nu = g_{0J_0} + 2 \wt f_\nu$, $\nu =1, 2$, and $g_{10} =
g_{0J_0} + 2 \wt f_1 + 2 \wt f_2$.  Thus
\begin{align}
  \begin{split}
    & L(\dx, a, b) = \, \dx + \sum_{j=0}^{J_0 - 1}a_{0j}g_{0j}(x)
    + (\delta_{0, J_0 + 1} - \delta_{0J_0 }) g_{0J_0}(x) + \\[-3mm]
    & + (\delta_{10} - \delta_{0, J_0 +1 })( g_{0J_0} + 2 \wt f_\nu
    )(x)
    + (\delta_{11} - \delta_{10}) ( g_{0J_0} + 2 \wt f_1 + 2 \wt f_2
    )(x) + \sum_{j=1}^{J_1}a_{1j}g_{1j}(x) =
  \end{split}\notag\displaybreak[0]\\[-3mm]
  \begin{split}
    & = \, \dx + \sum_{j=0}^{J_0 - 1}a_{0j}g_{0j}(x) + (\delta_{11} -
    \delta_{0J_0 }) g_{0J_0}(x) +
    2 (\delta_{11} - \ep_1) \wt f_1(x) + \\
    & + 2 (\delta_{11} - \ep_2) \wt f_2 (x) +
    \sum_{j=1}^{J_1}a_{1j}g_{1j}(x).  \label{eq:L}
  \end{split}
\end{align}
\begin{remark}\label{rem:L}
Equation \ref{eq:L} shows that in $ L(\dx, a, b)$ we have the same
first order expansion, whatever the sign of $\ep_2 - \ep_1$. This
proves that the finite-dimensional problem (FP) is $C^1$.
\end{remark}

Let $\wh\beta := \beta\circ \wh S_T$ and $\wh\gamma := \alpha
+\wh\beta $.  Then the cost \eqref{eq:subcost} can be written as
\[
J(x, a, b) = \a(x) + \beta\circ S_T(x, a, b) = \a(x) + \wh\beta\circ
\zeta_T(x, a, b)
\]
By PMP $\ud\wh\gamma(\wh x_0) = 0$, thus the first variation of $J$ at
$(x, a, b) = (\hat x_0, 0, 0)$ is given by
\begin{equation*}
  J^\prime (\dx, a, b) =  \Big( b h_\nu + \dsum_{i=0}^1
  \dsum_{j=0}^{J_i}a_{ij} g_{ij} \Big)\cdot\wh\beta(\wh x_0)
\end{equation*}
which, by \eqref{eq:L}, does not depend on $\nu$, i.e. it does not
depend on the sign of $\ep_2 - \ep_1$.

On the other hand, the second order expansion of $\zeta^\nu_T(x,
\cdot, \cdot )$ at $(a,b) = (0,0)$ is given by
\begin{equation*}
  \begin{split}
    \exp & \Bigg(b h_\nu + \dsum_{i=0}^1\dsum_{j=0}^{J_i}a_{ij}g_{ij}
    + \dfrac{1}{2}\Bigg\{ \dsum_{j=0}^{J_0}a_{0j}\Big[ g_{0j},
    \dsum_{s= j+1}^{J_0} a_{0s} g_{0s} + b h_\nu + \dsum_{j=0}^{J_1}
    a_{1j} g_{1j} \Big] +
    \\
    & + b \Big[ h_\nu, \dsum_{j=0}^{J_1} a_{1j} g_{1j} \Big] +
    \dsum_{j=0}^{J_1} a_{1j}\Big[ g_{1j}, \dsum_{s= j+1}^{J_1} a_{1s}
    g_{1s} \Big] \Bigg\} \Bigg)( x ).
  \end{split}
\end{equation*}
where $\nu =1$ if $\ep_1 \leq \ep_2$, $\nu =2$ otherwise.  Proceeding
as in \cite{ASZ02b} we get for all $(\dx,a,b)\in
\ker J^\prime$,
\begin{align*}
  & J_\nu\se [(\dx, a, b)]^2 = \dfrac{1}{2}\Big\{ \ud^2 \wh\gamma (\wh
  x_0)[\dx]^2 
  + 2 \, \dx \cdot \Big( \dsum_{i=0}^{1} \dsum_{j=0}^{J_i} a_{ij}\, g_{ij} + b h_\nu \Big) \cdot \wh\beta (\wh x_0)   +   \displaybreak[0]\nonumber\\
  & + \Big( \dsum_{i=0}^{1} \dsum_{j=0}^{J_i}a_{ij}\, g_{ij} + b h_\nu
  \Big)^2 \cdot \wh\beta (\wh x_0) 
  + \dsum_{j=0}^{J_0} \dsum_{i=0}^{j-1} a_{0i}a_{0j} [g_{0i},
  g_{0j}] \cdot \wh\beta (\wh x_0) + \displaybreak[0]\nonumber\\
  & + b \dsum_{i=0}^{J_0} a_{0i} [g_{0i}, h_\nu] \cdot \wh\beta (\wh
  x_0) 
  + \dsum_{j=0}^{J_1} a_{1j} \Big( \dsum_{i=0}^{J_0} a_{0i}
  [g_{0i}, g_{1j}] + b [h_\nu, g_{1j}] +\displaybreak[0]\nonumber\\
  & + \dsum_{i=0}^{j-1} a_{1i} [g_{1i}, g_{1j}] \Big) \cdot \wh\beta
  (\wh x_0) \Big\} \nonumber
\end{align*}
where, again, $\nu =1$ if $\ep_1 \leq \ep_2$, $\nu =2$ otherwise.
\begin{remark}  \label{rem:2var}
The previous formula clearly shows that $J_1\se = J_2\se$ on 
$\{(\dx, a, b) \colon b = 0 \} $, i.e.~on $\{(\dx, \delta, \ep) \colon
\ep_1 = \ep_2\}$. The second variation is $J_1\se$ if $\ep_1 \leq
\ep_2$, $J_2\se$ otherwise. Its coercivity means that both $J_1\se$
and $J_2\se$ are coercive quadratic forms.
\end{remark}
\begin{remark}
  Isolating the addenda where $a_{0J_0}$, $b$, $a_{10}$ appear, as in
  \eqref{eq:L}, one can easily see that $J\se_1 = J\se_2$ if and only
  if $[\wt f_1, \wt f_2]\cdot\wh\beta(\wh x_0) = 0$, i.e. if and only
  if $\scal{\wh\lambda(\hat\tau)}{[f_1, f_2](\hat x_d)} = 0$. In other
  words: problem (FP) is twice differentiable at $(x, \delta, \ep) =
  (\hat x_0, 0, 0)$ if and only if $\scal{\wh\lambda(\hat\tau)}{[f_1,
    f_2](\hat x_d)} = 0$.
\end{remark}

The bilinear form associated to each $J\se_\nu$ is given by
\begin{align}\label{eq:bili}
  & J_\nu\se \left( (\dx, a, b) , (\dy, c, d) \right) =
  \dfrac{1}{2}\Bigg\{ \ud^2 \wh\gamma (\wh x_0) ( \dx, \dy ) + \displaybreak[0]\\[-1mm]
  & + \dy \cdot \Big( \dsum_{i=0}^{J_0}a_{0i}\, g_{0i} + b h_\nu +
  \dsum_{i=0}^{J_1}a_{1i}\,
  g_{1i}  \Big) \cdot \wh\beta (\wh x_0)   +    \displaybreak[0]\nonumber\\[-1mm]
  & + \dx \cdot \Big( \dsum_{i=0}^{J_0}c_{0i}\, g_{0i} + d h_\nu +
  \dsum_{i=0}^{J_1}c_{1i}\,
  g_{1i} \Big) \cdot \wh\beta (\wh x_0)    +    \displaybreak[0]\nonumber\\[-1mm]
  & + \Big( \dsum_{i=0}^{J_0}c_{0i}\, g_{0i} + d h_\nu +
  \dsum_{i=0}^{J_1}c_{1i}\, g_{1i} \Big) \cdot \Big(
  \dsum_{i=0}^{J_0}a_{0i}\, g_{0i} + b h_\nu +
  \dsum_{i=0}^{J_1}a_{1i}\,
  g_{1i} \Big) \cdot \wh\beta (\wh x_0)   + \displaybreak[0]\nonumber\\[-1mm]
  & + \dsum_{j=0}^{J_0} \dsum_{i=0}^{j-1} a_{0i}c_{0j} [g_{0i},
  g_{0j}] \cdot \wh\beta (\wh x_0) + d \dsum_{i=0}^{J_0} a_{0i}
  [g_{0i}, h_\nu] \cdot \wh\beta (\wh x_0)  +  \displaybreak[0]\nonumber\\[-1mm]
  & + \dsum_{j=0}^{J_1} c_{1j} \Big( \dsum_{i=0}^{J_0} a_{0i} [g_{0i},
  g_{1j}] + d [h_\nu, g_{1j}] + \dsum_{i=0}^{j-1} a_{1i} [g_{1i},
  g_{1j}] \Big) \cdot \wh\beta (\wh x_0) \Bigg\} \nonumber
\end{align}
By assumption, for each $\nu = 1,2$, $J\se_\nu$ is positive definite
on
\begin{multline*}
   \qquad \cN_0  := \Big\{ (\dx, a, b) \in T_{\wh x_0}N_0 \times
    \R^{J_0 + J_1+ 2} \times \R \colon  \\[-2mm]
    b + \dsum_{i=0}^1\dsum_{j=0}^{J_i}a_{ij} = 0,   \quad 
     L(\dx, a, b) \in T_{\wh x_f} N_f \Big\}. \qquad
\end{multline*}
Again following the procedure of \cite{ASZ02b} we may redefine
$\alpha$  by adding a suitable second-order penalty at $\wh x_0$ (see
e.g.~\cite{Hes51}, Theorem 13.2) and we may assume that each second
variation $J\se_\nu$ is positive definite on 
\begin{multline*}
  \qquad \cN := \Big\{ (\dx, a, b) \in T_{\wh x_0}M \times \R^{J_0
    + J_1 + 2}
  \times \R\colon \quad  \\[-2mm]
  b + \dsum_{i=0}^1\sum_{j=0}^{J_i}a_{ij} = 0 , \quad
  L(\dx, a, b)  \in T_{\wh x_f} N_f\Big\}, \qquad
\end{multline*}
i.e.~we can remove the constraint on the initial point of admissible
trajectories. \\ 
Let
\[
\Lambda := \{ \ud\a(x) \colon x \in M \}
\]
and introduce the anti-symplectic isomorphism $i$ as in
\cite{ASZ02b},
\begin{equation}
  \label{eq:isoanti}
  i \colon \dpx \in T^*_{\wh x_0} M \times T_{\wh x_0}M
  \mapsto -\delta p + \ud\, (-\wh\beta)_*\dx 
  \in T\left( T^*M\right).
\end{equation}
Define $ \vG\se_{ij} = i^{-1}\left( \vG_{ij}(\lo)\right)$, $\vH\se_\nu
= i^{-1}\left( \vH_\nu(\lo)\right)$. The Hamiltonian fields
$\vG\se_{ij}$ and $\vH\se_\nu$ are associated to the following linear
Hamiltonians defined in $T^*_{\wh x_0}M \times T_{\wh x_0}M$
\begin{align}
  \label{eq:hamisecG}
  \begin{split}
    G\se_{ij} (\omega, \dx) = \scal{\omega}{ g_{ij}(\wh x_0)} + \dx
    \cdot g_{ij} \cdot \wh\beta(\wh x_0)
  \end{split}
  \displaybreak[0]\\[1mm]
  \label{eq:hamisecH}
  \begin{split}
    H\se_\nu (\omega, \dx) = \scal{\omega}{h_\nu(\wh x_0)} + \dx \cdot
    h_\nu \cdot \wh\beta(\wh x_0).
  \end{split}
\end{align}
Moreover $ L\se_0 := i^{-1}T_{\lo}\Lambda = \left\{ \dl \in T^*_{\wh
    x_0}M \times T_{\wh x_0}M \colon \dl = \left(
    -D^2\wh\gamma(\wh x_0)(\dx, \cdot)\right) \right\}$. With such notation, the bilinear form
$J\se_\nu$ associated to the second variation can be written in a
rather compact form, see, e.g.~\cite{ASZ02b} or \cite{Pog06}.

\noindent For any $\de := (\dx, a, b) \in \cN$ let
\begin{align*}
  & \omega_0 := -D^2\wh\gamma(\wh x_0)(\dx, \cdot), \quad \dl :=
  (\omega_0, \dx )= i^{-1}
  \left(\ud\a_*\dx \right), \displaybreak[0]\\[-1mm]
  & (\omega_\nu, \dx_\nu) := \dl + \dsum_{i=0}^1\dsum_{j=0}^{J_i}
  a_{ij}\vG\se_{ij} + b\, \vH\se_\nu \; \text{ and } \; \dl_\nu :=
  (\omega_\nu, \dx_\nu).
\end{align*}
Then $J\se_\nu$ can be written as
\begin{equation}
  \label{eq:bilihami}
  \begin{split}
    J_\nu\se &\Big((\dx, a, b) , (\dy, c, d) \Big) = -
    \scal{\omega_\nu}{\dy + \dsum_{s=0}^{J_0}
      c_{0s} g_{0s} + d\, h_\nu + \dsum_{s=0}^{J_1}c_{1s} g_{1s} } \\
    &+ \dsum_{j=0}^{J_0} c_{0j}\, G\se_{0j}\Big( \dl +
    \dsum_{s=0}^{j-1} a_{0s}\vG\se_{0s}\Big)
    + d\, H\se_\nu \Big( \dl + \dsum_{s=0}^{J_0} a_{0s}\vG\se_{0s}\Big) \\
    & + \dsum_{j=0}^{J_1}c_{1j} G\se_{1j} \Big( \dl +
    \dsum_{s=0}^{J_0} a_{0s}\vG\se_{0s} + b \vH\se_\nu +
    \dsum_{s=0}^{j-1} a_{1s}\vG\se_{1s} \Big)
  \end{split}
\end{equation}
We shall study the positivity of $J\se_\nu$ as follows: consider
\[ 
V := \Big\{ (\dx, a, b) \in \cN\colon L(\dx, a, b) = 0 \Big\} 
\]
and the sequence
\[
V_{01} \subset \ldots \subset V_{0J_0} \subset V_{10} \subset \ldots
\subset V_{1J_1} = V 
\]
of sub-spaces of $V$, defined as folllows
\begin{align*}
  & V_{0j} := \{ (\dx, a, b) \in V\colon
  a_{0s} =0 \quad \forall s = j+1, \ldots, J_0, \; a_{1s} = 0  \}  \\
  & V_{1j} := \{ (\dx, a, b) \in V \colon a_{1s} =0 \quad
  \forall s = j+1, \ldots, J_1 \}.
\end{align*}
Observe that $V_{0j}^1 = V_{0j}^2$ for any $j=0, \ldots, J_0$, so we
denote these sets as $V_{0j}$. Moreover
\begin{equation*}
  \dim\left(V_{0j}\cap V_{0, j-1}^{{\perp_{J\se_\nu}}} \right)
  = \dim\left(V_{1k}\cap V_{1, k-1}^{{\perp_{J\se_\nu}}} \right) = 1, \quad
  \dim\left(V_{10}\cap V_{0 J_0}^{{\perp_{J\se_\nu}}} \right) = 2
\end{equation*}
for any $j=2, \ldots, J_0$, $k=0, \ldots, J_1$ and $\nu =1, 2$ and
$J\se_\nu$ is positive definite on $\cN$ if and only if it is positive
definite on each $V_{ij}\cap V_{i, j-1}^{\perp_{J\se_\nu}}$,
$V_{10}\cap V_{0J_0}^{\perp_{J\se_\nu}}$ and $\cN\cap V^{\perp_{J\se_\nu}}$. 

As in \cite{ASZ02b} one can prove a characterization, in terms of the
maximized flow, of the intersections above. We state here such
characterization without proofs which can be found in the
aforementioned paper.
\begin{lemma}\label{lem:v0jperp} 
  Let $j=1, \ldots, J_0$ and $\de = (\dx, a,b) \in V_{0j}$. Assume
  $J\se_\nu$ is positive definite on $V_{0, j-1}$.  Then $\de
  \in V_{0j}\cap V_{0, j-1}^{\perp_{J\se_\nu}}$ if and only if
  \begin{equation}
    \label{eq:v0jperp}
    G\se_{0s}(\dl +
    \displaystyle\sum_{r=0}^{s-1}a_{0r}\vG\se_{0r}) =
    G\se_{0,j-1}(\dl +
    \displaystyle\sum_{s=0}^{j-2}a_{0s}\vG\se_{0s}) 
    \, , \quad \forall \, s = 0,\ldots,j-2
  \end{equation}
  i.e. if and only if
  \begin{equation}\label{eq:carat0j}
    a_{0s} = \scal{\ud\left( \theta_{0, s+ 1} -
        \theta_{0s}\right)(\lo)}{\ud\a_*\dx}  \quad \forall s = 0, \ldots, j-2.
  \end{equation}
  In this case
  \begin{equation}
    \label{eq:jsekk1}
    \begin{split}
      J\se_\nu[\de]^2 &= a_{0j} \left( G\se_{0j} - G\se_{0,j-1}
      \right)(\dl + \displaystyle\sum_{s=0}^{j-1}a_{0s}
      \vG\se_{0s}) = \\
      &= a_{0j} \, \bsi\Big({\dl + \displaystyle\sum_{s=0}^{j-1}a_{0s}
        \vG\se_{0s}},{ \vG\se_{0j} - \vG\se_{0,j-1}}\Big) \\
      & = - a_{0j} \, \bsi\Big({\ud\alpha_*\dx +
        \displaystyle\sum_{s=0}^{j-1}a_{0s} \vG_{0s}(\lo)}, {(
        \vG_{0j} - \vG_{0,j-1})(\lo)}\Big) .
    \end{split}
  \end{equation}
\end{lemma}
\begin{lemma} \label{lem:v10perp} Let $\nu =1, 2$ and $\de = (\dx,
  a,b) \in V_{10}$. Assume $J\se_\nu$ is positive definite on $V_{0, J_0}$.   
Then $\de \in V_{10}\cap V_{0 J_0}^{\perp_{J\se_\nu}}$ if and only if
  \begin{equation}
    \label{eq:v10perp}
    G\se_{0s}(\dl +
    \displaystyle\sum_{\mu=0}^{s-1}a_{0\mu}\vG\se_{0\mu}) =
    G\se_{0,J_0}(\dl +
    \displaystyle\sum_{s=0}^{J_0-1}a_{0s}\vG\se_{0s}) 
    \, , \quad \forall \, s = 0,\ldots, J_0 - 1
  \end{equation}
  i.e. if and only if
  \begin{equation}\label{eq:carat10}
    a_{0s} = \scal{\ud\left( \theta_{0, s+ 1} -
        \theta_{0s}\right)(\lo)}{\ud\a_*\dx}  \quad \forall s = 0, \ldots, J_0 - 1.
  \end{equation}
  In this case
  \begin{equation}\label{eq:varsecv10perp}
    \begin{split}    
      J\se_\nu[\de]^2 = & \, b \left( H\se_\nu - G\se_{0J_0} \right)(\dl +
      \displaystyle\sum_{s=0}^{J_0}a_{0s} \vG\se_{0s}) \\
      & + a_{10} \left( G\se_{10} - H\se_\nu\right)(\dl +
      \displaystyle\sum_{s=0}^{J_0}a_{0s} \vG\se_{0s} + b \vH\se_\nu )  = \\
      = & \, b \, \bsi\Big({\dl + \displaystyle\sum_{s=0}^{J_0}a_{0s}
        \vG\se_{0s}},  {\vH\se_\nu - \vG\se_{0,J_0}} \Big) + \\
      & + a_{10} \, \bsi \Big( {\dl +
        \displaystyle\sum_{s=0}^{J_0}a_{0s} \vG\se_{0s} + b
        \vH\se_\nu}, {\vG\se_{10} - \vH\se_\nu} \Big) = \\
      = & \, - b \, \bsi\Big({\ud\alpha_*\dx +
        \displaystyle\sum_{s=0}^{J_0}a_{0s}
        \vG_{0s}(\lo)},  {(\vH_\nu - \vG_{0,J_0})(\lo)} \Big) - \\
      & - a_{10} \, \bsi \Big( {\ud\alpha_*\dx +
        \displaystyle\sum_{s=0}^{J_0}a_{0s} \vG_{0s}(\lo) + b
        \vH_\nu(\lo)}, {(\vG_{10} - \vH_\nu)(\lo)} \Big) .
    \end{split}
  \end{equation}
\end{lemma}
\begin{lemma} \label{lem:v1jperp} Let $\nu =1, 2$, $j=1, \ldots, J_1$
  and $\de = (\dx, a,b) \in V_{1j}$. Assume
  $J\se_\nu$ is positive definite on $V_{1, j-1}$.   
  Then $\de \in V_{1j} \cap
  V_{1, j-1}^{\perp_{J\se_\nu}}$ if and only if
  \begin{multline*}
    \label{eq:v1jperp}
    G\se_{0s}(\dl + \displaystyle\sum_{i=0}^{s-1}a_{0i}\vG\se_{0i}) =
    G\se_{1,j-1}(\dl + \displaystyle\sum_{i=0}^{J_0}a_{0i}\vG\se_{0i}
    + b \vH\se_\nu +
    \dsum_{i=0}^{j-2}a_{1i}\vG\se_{1i})  = \\
    = H\se_\nu(\dl + \displaystyle\sum_{i=0}^{J_0}a_{0i}\vG\se_{0i} )
    = G\se_{1k}(\dl + \displaystyle\sum_{i=0}^{J_0}a_{0i}\vG\se_{0i} +
    b \vH\se_\nu
    + \dsum_{i=0}^{k-1}a_{1i}\vG\se_{1i}) \\
    \forall \, s = 0,\ldots, J_0 \quad \forall \, k = 0, \ldots, j-2
  \end{multline*}
  i.e. if and only if
  \begin{equation*}
    \begin{split}
      a_{0s} & = \scal{\ud \left( \theta_{0, s+ 1} - \theta_{0s}
        \right)
        (\lo)}{\ud\a_*\dx} \quad \forall s = 0, \ldots, J_0 \\
      b & = \scal{ \ud\left( \theta_{10} - \theta_{0, J_0 +
            1}\right)(\lo)}{\ud\a_*\dx} \\
      a_{1s} & = \scal{\ud\left( \theta_{1, s+ 1} -
          \theta_{1s}\right)(\lo)}{\ud\a_*\dx} \quad \forall s = 0,
      \ldots, j-2.
    \end{split}
  \end{equation*}
  In this case
  \begin{multline*}
    J\se_\nu[\de]^2 = \, a_{1j} \left( G\se_{1j} - G\se_{1,j-1}
    \right)(\dl + \displaystyle\sum_{s=0}^{J_0}a_{0s} \vG\se_{0s} + b
    \vH\se_\nu
    + \dsum_{i=0}^{j-1}a_{1i}\vG\se_{1i})  \\
    = a_{1j}\, \bsi \Big( \dl + \displaystyle\sum_{s=0}^{J_0}a_{0s}
    \vG\se_{0s} + b \vH\se_\nu + \dsum_{i=0}^{j-1}a_{1i}\vG\se_{1i} \,
    , \, \vG\se_{1j} - \vG\se_{1,j-1} \Big) \\
    = -a_{1j}\, \bsi \Big( \! \ud\alpha_*\dx +
    \displaystyle\sum_{s=0}^{J_0}\! a_{0s} \vG_{0s}(\lo) + b \vH_\nu
    (\lo) + \dsum_{i=0}^{j-1}\! a_{1i}\vG_{1i}(\lo) , ( \vG_{1j} -
    \vG_{1,j-1} )(\lo)\Big) .
  \end{multline*}
\end{lemma}
\begin{lemma}\label{lem:Njperp}
  Let $\nu =1, 2$ and $\de = (\dx, a,b) \in \cN$.  Assume
  $J\se_\nu$ is positive definite on $V_{1J_1}$.   Then $\de \in
  \cN \cap V_{1J_1}^{\perp_{J\se_\nu}}$ if and only if
  \begin{multline*}
    G\se_{0s}(\dl + \displaystyle\sum_{i=0}^{s-1}a_{0i}\vG\se_{0i}) =
    G\se_{1,J_1}(\dl + \displaystyle\sum_{i=0}^{J_0}a_{0i}\vG\se_{0i}
    + b \vH\se_\nu +
    \dsum_{i=0}^{J_1 - 1}a_{1i}\vG\se_{1i})   = \\
    = H\se_\nu(\dl + \displaystyle\sum_{i=0}^{J_0}a_{0i}\vG\se_{0i} )
    = G\se_{1k}(\dl + \displaystyle\sum_{i=0}^{J_0}a_{0i}\vG\se_{0i} +
    b \vH\se_\nu
    + \dsum_{i=0}^{k-1}a_{1i}\vG\se_{1i}) \\
    \forall \, s = 0,\ldots, J_0 \quad \forall \, k = 0, \ldots, J_1
  \end{multline*}
  i.e. if and only if $\de \in \cN$ and
  \begin{equation*}
    \begin{split}
      a_{0s} & = \scal{ \ud\left( \theta_{0, s+ 1} -
          \theta_{0s}\right)(\lo)}
      {\ud\a_*\dx}  \quad \forall s = 0, \ldots, J_0 \\
      b & = \scal{\ud\left( \theta_{10} - \theta_{0, J_0 +
            1}\right)(\lo)}{\ud\a_*\dx} \\
      a_{1s} & = \scal{\ud\left( \theta_{1, s+ 1} -
          \theta_{1s}\right)(\lo)}{\ud\a_*\dx} \quad \forall s = 0,
      \ldots, J_1 - 1.
    \end{split}
  \end{equation*}
  In this case
  \begin{equation*}
    \begin{split}
      J\se_\nu[\de]^2 & = \, - \scal{\omega_\nu}{\dx + \dsum_{i=0}^1
        \dsum_{s=0}^{J_i}
        a_{is} g_{is}(\wh x_0) + b\, h_\nu(\wh x_0) }  = \\
      & = \bsi \Big(\big(0, \dx + \dsum_{i=0}^1\dsum_{s=0}^{J_i}
      a_{is} g_{is}(\wh x_0) + b h_\nu(\wh x_0) \big) \, , \\
      & \hspace{4cm} -D^2\wh\gamma(\wh x_0)(\dx, \cdot) +
      \dsum_{i=0}^1 \dsum_{s=0}^{J_i} a_{is} \vG\se_{is} + b
      \vH\se_\nu \Big) = \\
      & = - \bsi \Big(\ud\,(-\wh\beta)_* \big( \dx +
      \dsum_{i=0}^1\dsum_{s=0}^{J_i}
      a_{is} g_{is}(\wh x_0) + b h_\nu (\wh x_0) \big) \, , \\
      & \hspace{45mm} \ud\alpha_* \dx + \dsum_{i=0}^1
      \dsum_{s=0}^{J_i} a_{is} \vG_{is}(\lo) + b \vH_{\nu}(\lo) \Big)
      .
    \end{split}
  \end{equation*}
\end{lemma}
\section{The invertibility of the flow}
\label{sec:invert}
We are now going to prove that  the map
\[
\id \times \pi\cH \colon (t, \ell) \in [0, T] \times \Lambda \mapsto
(t, \pi\cH_t(\ell)) \in [0, T] \times M
\]
is one-to-one onto a neighborhood of the graph of $\wh\xi$.
Since the time interval $[0, T]$ is compact and by the properties of
flows, it suffices to show that $\pi\cH_{\hat\theta_{ij}}$, $i=1, 2$,
$j=1, \ldots, J_i$ and $\pi\cH_{\hat\tau}$ are one-to-one onto a
neighborhood of $\wh\xi(\hat\theta_{ij})$ and $\wh\xi(\hat\tau)$ in $M$,
respectively.

The proof of the invertibility at the simple switching times
$\hat\theta_{0j}$, $j=1, \ldots, J_0$ my be carried out either as in
\cite{ASZ02b} or by means of  Clarke's inverse function theorem (see
\cite[Thm 7.1.1.]{Cla83}), while the invertibility at the double
switching time and at the simple switching times $\hat\theta_{1j}$,
$j=1, \ldots, J_1$ will be proved by means of  Clarke's
inverse function theorem or by means of topological methods (see
Theorem \ref{thm:topo}) according to the dimension of the kernel of
$\ud\,(\tau_1 - \tau_2)\vert_{T_\lo\Lambda}$. 

For the sake of uniformity with the others switching
times, for the simple switching times $\hat\theta_{0j}$, $j=1,
\ldots, J_0$ and  we give here the proof based on Clarke's
inverse function theorem. Namely, we consider the expressions of
$\pi\cH_{\hat\theta_{0j}}(\ell)$, which are different according to
whether $\theta_{0j}(\ell)$ is greater than or smaller than
$\hat\theta_{0j}$. We write the linearization of such expressions and
their convex combinations. Finally, using the coercivity of the
second variation on $V_{0j}$ we prove that all their convex
combinations are one--to--one.

The flow $\cH_{\hat\theta_{0j}}$ at time
$\hat\theta_{0j}$, associated to the maximized Hamiltonian defined in
equation \eqref{eq:maxfl}, has the following expression:
\[
\cH_{\hat\theta_{0j}}(\ell) = \begin{cases}
  \exp \hat\theta_{0j}\vK_{0,j-1}(\phi_{0,j-1}(\ell)) & \text{if } \theta_{0j}(\ell) > \hat\theta_{0j}\\
  \exp(\hat\theta_{0j} - \theta_{0j}(\ell))\vK_{0j} \circ \exp
  \theta_{0j}(\ell)\vK_{0,j-1}(\phi_{0,j-1}(\ell)) & \text{if }
  \theta_{0j}(\ell) < \hat\theta_{0j} .
\end{cases}
\]
\begin{lemma}
  \label{lem:maxflow0j}
  Let $j \in \{1, \ldots, J_0\}$. Define
  \begin{align*}
    & A_{0j} \colon \dl \in T_\lo\Lambda \mapsto \pi_*\exp\hat\theta_{0j}\vK_{0, j-1 \, *}\varphi_{0, j-1 \, *}\dl \in T_{\wh\xi(\hat\theta_{0j})}M \\
    & B_{0j} \colon \dl \in T_\lo\Lambda \mapsto A_{0j}\dl -
    \scal{\ud\theta_{0j}(\lo)}{\dl}\big( k_{0j} - k_{0, j-1}
    \big)\vert_{\wh\xi(\hat\theta_{0j}) } \in
    T_{\wh\xi(\hat\theta_{0j})}M
  \end{align*}
  Then, for any $t \in [0, 1]$, the map
  \begin{equation*}
    tA_{0j} + (1-t)B_{0j} \colon T_\lo\Lambda \to T_{\wh\xi(\hat\theta_{0j})}M 
  \end{equation*}
  is one-to-one.
\end{lemma}
\begin{proof}
  Let $t \in [0,1]$ and let $\dl \in T_\lo\Lambda$ such that $(tA_{0j}
  + (1-t)B_{0j})(\dl) = 0 $. We need to show that $\dl$ is null. From
  formula \eqref{eq:linphi0j} it follows that $\dl$ is in $
  \ker(tA_{0j} + (1-t)B_{0j})$ if and only if
  \begin{equation}\label{eq:clarke1}
    \pi_*\wh\cH_{\hat\theta_{0j}\, *}  \Delta_{0, j-1} \dl = 0.
  \end{equation}
  Let $\dx := \pi_*\dl$, so that $\dl = \ud\alpha_*\dx$.  Equation
  \eqref{eq:clarke1} is equivalent to
  \begin{multline}
    \label{eq:clarke2}
    \dx + \sum_{s=1}^{j-2} \scal{\ud(\theta_{0, s+1} -
      \theta_{0s})(\lo)}{\dl}
    g_{0s}(\wh x_0) + \\
    + \big( t \scal{\ud\theta_{0j}(\lo)}{\dl} - \scal{\ud\theta_{0,
        j-1}(\lo)}{\dl}\big) g_{0, j-1}(\wh x_0) - t
    \scal{\ud\theta_{0j}(\lo)}{\dl} g_{0j}(\wh x_0) = 0.
  \end{multline}
  Let $\de := (\dx, a, b)$ such that
  \begin{align*}
    & a_{0s} =  \scal{\ud(\theta_{0, s+1} - \theta_{0s})(\lo)}{\dl} \quad s = 0, \ldots, j-2 \\
    & a_{0, j-1} = t \scal{\ud\theta_{0j}(\lo)}{\dl} -
    \scal{\ud\theta_{0, j-1}(\lo)}{\dl} \\
    & a_{0j} = - t \scal{\ud\theta_{0j}(\lo)}{\dl} \\
    & a_{0s} = b = a_{1r}=0 \quad s = j+1, \ldots, J_0, \ r = 0,
    \ldots, J_1.
  \end{align*}
  There are three possible cases: \\
  {\em a) } If $t=0$, then $\de \in V_{0,j-1} \cap V_{0,
    j-1}^{{\perp_{J\se_\nu}}} = \{ 0 \}$, because of the coercivity of
  $J\se_\nu$.  \\
  {\em b) } If $t=1$, then $\de \in V_{0j} \cap
  V_{0j}^{{\perp_{J\se_\nu}}} = \{ 0 \}$, because of the coercivity of
  $J\se_\nu$.  In both cases we thus have $\dx = 0$, so that $\dl =
  \ud\alpha_*\dx$
  is also null. \\
  {\em c) } If $t \in (0, 1)$, then $\de \in V_{0j} \cap V_{0,
    j-1}^{{\perp_{J\se_\nu}}}$. Therefore, applying \eqref{eq:jsekk1}
  we get
  \begin{align*}
    0 & < J\se_\nu[\de]^2 \, = \, t \scal{\ud\theta_{0j}(\lo)}{\dl}\, \bsi
    \big( \dl + \dsum_{s=0}^{j-2}
    \scal{\ud(\theta_{0, s+1} - \theta_{0s})(\lo)}{\dl}\vG_{0s}(\lo) + \\
    + & \big( t \scal{\ud\theta_{0j}(\lo)}{\dl} - \scal{\ud\theta_{0,
        j-1}(\lo)}{\dl} \big)\vG_{0, j-1}(\lo) \, , \, (\vG_{0j} - \vG_{0, j-1})(\lo)\big) = \\
    = \, & t \, \scal{\ud\theta_{0j}(\lo)}{\dl} \, \bsi \big(
    \Delta_{0, j-1}\dl
    +   t  \, \scal{\ud\theta_{0j}(\lo)}{\dl} \vG_{0, j-1}(\lo) \, , \, (\vG_{0j} - \vG_{0, j-1})(\lo)\big) = \\
    = & -t\, (1 -t) \scal{\ud\theta_{0j}(\lo)}{\dl}^2
    \dueforma{\vG_{0, j-1}}{\vG_{0j}}(\lo),
  \end{align*}
  a contradiction.
\end{proof} 

Lemma \ref{lem:maxflow0j} implies that Clarke's Generalized Jacobian
of the map $\pi\cH_{\hat\theta_{0j}}$ at $\lo$ is of maximal
rank. Therefore, by Clarke's inverse function theorem (see \cite[Thm
7.1.1.]{Cla83}) the map $\pi\cH_{\hat\theta_{0j}}$ is locally
invertible about $\lo$ with Lipschitz continuous inverse. Hence the
map
\begin{equation}
  \label{eq:invappl}
  \psi \colon (t, \ell) \in [0, T] \times \Lambda
  \mapsto \left( t, \pi  \cH_t(\ell)\right) \in [0, T] \times M
\end{equation}
is locally invertible about $[0, \hat\tau - \e] \times \big\{\lo
\big\}$.  In fact, $\psi$ is locally one-to-one if and only if $\pi
\cH_t$ is locally one-to-one in $\lo$ for any $t$. On the other hand
$\pi \cH_t$ is locally one-to-one for any $t < \hat\tau$ if and only
if it is one-to-one at any $\hat\theta_{0j}$.

We now show that such procedure can be carried out also on
$[\hat\tau - \e, T] \times \big\{\lo \big\}$, so that $\psi$ will turn
out to be locally invertible from a neighborhood $[0, T] \times \cO
\subset [0, T] \times\Lambda$ of $[0, T] \times \big\{\lo \big\}$ onto
a neighborhood $\cU \subset [0, T] \times M$ of the graph $\wh\Xi$ of
$\wh\xi$.  The first step will be proving the invertibility of
$\pi\cH_{\hat\tau}$ at $\lo$.

In a neighborhood of $\lo$, $\pi\cH_{\hat\tau}$ has the following
piecewise representation:
\begin{enumerate}
\item if $\min\big\{ \tau_1(\ell), \; \tau_2(\ell)
  \big\}\geq\widehat\tau$, then %
  $\quad \pi\cH_{\hat\tau}(\ell) = \exp \hat\tau\vK_{0
    J_0}\circ\phi_{0 J_0}(\ell)$,
\item if $\min\big\{\tau_1(\ell),\tau_2(\ell)\big\}=
  \tau_1(\ell)\leq\hat\tau\leq\theta_{10}(\ell)$, then
  \[
  \pi\cH_{\hat\tau}(\ell) = \exp(\hat\tau -\tau_1(\ell))\vK_1 \circ
  \exp\tau_1(\ell)\vK_{0 J_0}\circ\phi_{0 J_0}(\ell),
  \]
\item if $\min\big\{\tau_1(\ell),\tau_2(\ell)\big\}=
  \tau_2(\ell)\leq\hat\tau\leq\theta_{10}(\ell)$, then
  \[
  \pi\cH_{\hat\tau}(\ell) = \exp(\hat\tau -\tau_2(\ell))\vK_2 \circ
  \exp\tau_2(\ell)\vK_{0 J_0}\circ \phi_{0 J_0}(\ell),
  \]
\item if $\min\big\{\tau_1(\ell),\tau_2(\ell)\big\}=
  \tau_1(\ell)\leq\theta_{10}(\ell)\leq\hat\tau$, then
  \begin{multline*}
    \pi\cH_{\hat\tau}(\ell) =\exp(\hat\tau\vK_{1 0})\circ\psi_{1
      0}(\ell)= \exp(\hat\tau - \theta_{1 0}(\ell))\vK_{1 0}\circ \\
    \circ\exp ( \theta_{1 0}(\ell) - \tau_1(\ell))\vK_1 \circ
    \exp\tau_1(\ell)\vK_{0 J_0}\circ \phi_{0 J_0}(\ell),
  \end{multline*}
\item if $\min\big\{\tau_1(\ell),\tau_2(\ell)\big\}=
  \tau_2(\ell)\leq\theta_{10}(\ell)\leq\hat\tau$, then
  \begin{multline*}
    \pi\cH_{\hat\tau}(\ell) =\exp(\hat\tau\vK_{1 0})\circ\psi_{1
      0}(\ell)= \exp(\hat\tau-\theta_{1 0}(\ell))\vK_{1 0}\circ \\
    \circ\exp(\theta_{1 0}(\ell) - \tau_2(\ell))\vK_2 \circ
    \exp\tau_2(\ell)\vK_{0 J_0}\circ \phi_{0 J_0}(\ell).
  \end{multline*}
\end{enumerate}
The invertibility of $\pi\cH_{\hat\tau}$ will be proved by means of
two different arguments: in the generic case when $\ud\,(\tau_1 -
\tau_2)(\lo) \colon T_{\lo}\Lambda \to \R$ is not identically zero, we
will use the topological argument of Theorem \ref{thm:topo} in the
Appendix; whereas, in the opposite case we will apply Clarke's inverse
function theorem \cite[Thm 7.1.1.]{Cla83}, as in the case of simple
switches. In particular, in the special case when
$\ud\tau_1(\lo)\vert_{T_{\lo}\Lambda} \equiv
\ud\tau_2(\lo)\vert_{T_{\lo}\Lambda} \equiv 0$ we will prove that
$\pi\cH_{\hat\tau}$ is indeed differentiable at $\lo$.

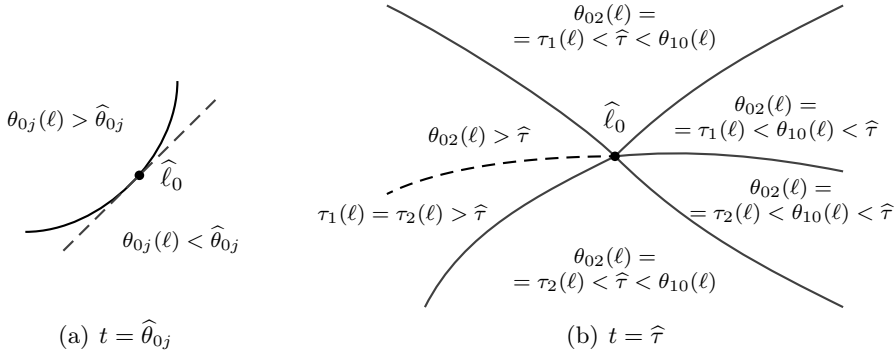
\begin{figure}[ht!]
\subfigure[$t=\wh\theta_{0j}$]{
      \begin{pspicture}(0,4)(4.5,0)
        \psbezier{-}(1,1)(2,1)(3,2)(3,3)
        \uput[l](2.5,2.5){\scriptsize $\theta_{0j}(\ell) > \wh\theta_{0j}$} 
        \psdot(2.5,1.75)
        \uput[r](2.1,0.9){\scriptsize $\theta_{0j}(\ell) <\wh\theta_{0j}$} 
        \uput[r](2.6,1.75){\small $\lo$} 
        \psline[linestyle=dashed,linecolor=darkgray]{-}(1.5,0.75)(3.5,2.75)
      \end{pspicture}}
\subfigure[$t=\wh\tau$]{
      \begin{pspicture}(1,4)(9,2) %
        \psbezier[linecolor=darkgray]{-}(2,6)(3,5.5)(4.5,4.5)(5,4)
        \psbezier[linestyle=dashed]{-}(2,3.5)(3,4)(4.7,4)(5,4)
        \psbezier[linecolor=darkgray]{-}(2.5,2)(3,3)(4,3.5)(5,4)
        \psbezier[linecolor=darkgray]{-}(8,2)(7,2.5)(6,3)(5,4)
        \psbezier[linecolor=darkgray]{-}(8,3.8)(7,4)(6,4.1)(5,4)
        \psbezier[linecolor=darkgray]{-}(8,6)(7,5.5)(6,5)(5,4)
        \uput[d](2.2,3.55){\scriptsize{$\tau_1(\ell) = \tau_2(\ell) > \wh\tau$}} %
        \uput[ul](4,4){\scriptsize{$\theta_{02}(\ell) > \wh\tau$}} %
        \uput[u](5,5.2){\scriptsize{$ \begin{matrix}\theta_{02}(\ell)=\\ =\tau_1(\ell)
            <\wh\tau < \theta_{10}(\ell)\end{matrix}$}} %
        \uput[d](5,3){\scriptsize{$ \begin{matrix}\theta_{02}(\ell)=\\ =\tau_2(\ell)
            <\wh\tau < \theta_{10}(\ell)\end{matrix}$}} %
        \uput[r](5.8,3.4){\scriptsize{$\begin{matrix}\theta_{02}(\ell)=\\ =\tau_2(\ell) <
            \theta_{10}(\ell) < \wh\tau\end{matrix}$}} %
        \uput[r](5.6,4.5){\scriptsize{$\begin{matrix}\theta_{02}(\ell)=\\ =\tau_1(\ell) <
            \theta_{10}(\ell) < \wh\tau\end{matrix}$}} %
        \uput[u](5,4.15){\small $\lo$}
        \psdot(5,4)
      \end{pspicture}
}
\caption{Local behaviour of $\cH_t$ near $\lo$ at a
simple switching time and at the double one.}
\end{figure}

In all cases we need to write the piecewise linearized map
$(\pi\cH_{\hat\tau})_*$.
\begin{subequations}\label{eq:lintau}
  \begin{enumerate}
  \item Let $M^0 = \{ \dl \in T_\lo\Lambda \colon
    \min\{\scal{\ud\tau_1(\lo)}{\dl}, \; \scal{\ud\tau_2(\lo)}{\dl} \}
    \geq 0 \}$. Then
    \begin{equation}
      \label{eq:lin1}
      (\pi\cH_{\hat\tau})_*\dl  = L^0\dl := (\exp\hat\tau k_{0 J_0 })_*\pi_*\phi_{0 J_0\, *}\dl \qquad\qquad\qquad \forall \dl \in M^0
    \end{equation}
  \item Let $M^{11} := \{ \dl \in T_\lo\Lambda \colon
    \scal{\ud\tau_1(\lo)}{\dl} \leq 0 \leq
    \scal{\ud\theta^1_{10}(\lo)}{\dl} , \ \scal{\ud\tau_1(\lo)}{\dl}
    \leq \scal{\ud\tau_2(\lo)}{\dl} \}$. Then
    \begin{multline}\label{eq:lin2}
      (\pi\cH_{\hat\tau})_*\dl = L^{11}\dl := - 2
      \scal{\ud\tau_1(\lo)}{\dl} f_1(\wh x_{\hat\tau}) +\exp(\hat\tau
      k_{0 J_0})_*\pi_*\phi_{0 J_0 *}\dl \\ \forall \dl \in M^{11}
    \end{multline}
  \item Let $M^{21} := \{ \dl \in T_\lo\Lambda \colon
    \scal{\ud\tau_2(\lo)}{\dl} \leq 0 \leq
    \scal{\ud\theta^2_{10}(\lo)}{\dl} , \ \scal{\ud\tau_2(\lo)}{\dl}
    \leq \scal{\ud\tau_1(\lo)}{\dl}\}$.  Then
    \begin{multline}\label{eq:lin3}
      (\pi\cH_{\hat\tau})_*\dl = L^{21}\dl := - 2
      \scal{\ud\tau_2(\lo)}{\dl} f_2(\wh x_{\hat\tau}) +\exp(\hat\tau
      k_{0 J_0})_*\pi_*\phi_{0 J_0 *}\dl \\ \forall \dl \in M^{21}
    \end{multline}
  \item Let $M^{12} := \{ \dl \in T_\lo\Lambda \colon
    \scal{\ud\tau_1(\lo)}{\dl} \leq \scal{\ud\theta^1_{10}(\lo)}{\dl}
    \leq 0 , \ \scal{\ud\tau_1(\lo)}{\dl} \leq
    \scal{\ud\tau_2(\lo)}{\dl} \} $. Then
    \begin{multline}
      \label{eq:lin4}
      (\pi\cH_{\hat\tau})_*\dl = L^{12}\dl := - 2
      \scal{\ud\theta_{10}^1(\lo)}{\dl}
      f_2 (\wh x_{\hat\tau}) - \\
      - 2 \scal{\ud\tau_1(\lo)}{\dl} f_1(\wh x_{\hat\tau})
      +\exp(\hat\tau k_{0 J_0})_*\pi_*\phi_{0 J_0 *}\dl \qquad \forall
      \dl \in M^{12}
    \end{multline}
  \item Let $M^{22} := \{ \dl\in T_\lo\Lambda \colon
    \scal{\ud\tau_2(\lo)}{\dl} \leq \scal{\ud\theta^2_{10}(\lo)}{\dl}
    \leq 0 , \ \scal{\ud\tau_2(\lo)}{\dl} \leq
    \scal{\ud\tau_1(\lo)}{\dl} \}$. Then
    \begin{multline}
      \label{eq:lin5}
      (\pi\cH_{\hat\tau})_*\dl = L^{22}\dl :=
      - 2 \scal{\ud\theta_{10}^2(\lo)}{\dl} f_1 (\wh x_{\hat\tau}) - \\
      - 2 \scal{\ud\tau_2(\lo)}{\dl} f_2(\wh x_{\hat\tau})
      +\exp(\hat\tau k_{0 J_0})_*\pi_*\phi_{0 J_0 *}\dl \qquad \forall
      \dl \in M^{22}
    \end{multline}
  \end{enumerate}
\end{subequations}

\begin{lemma}\label{lem:equiori}
  The piecewise linearized maps \eqref{eq:lintau} have the same
  orientation in the following sense: given any basis of
  $T_{\lo}\Lambda_0$ and any basis of $T_{\wh\xi(\hat\tau)}M$, the
  determinants of the matrices associated to the linear maps $L^0$,
  $L^{\nu j}$, $\nu, j = 1,2$, in such bases, have the same sign.
\end{lemma}
\begin{proof}
  The proof is given by means of Lemma \ref{lem:lemalg}. We show
  that for any $\dl_1$, $\dl_2 \in T_{\lo}\Lambda$ and $\nu =1, 2$ the following claims hold: \\[2mm]
  {\itshape Claim 1. } If $ \scal{\ud\tau_\nu(\lo)}{\dl} < 0 <
  \scal{\ud\tau_\nu(\lo)}{\dl_1}$ then $L^0(\dl_1) \ \neq
  L^{\nu 1}_{\wh\tau}(\dl_2)$, i.e.
  \begin{equation*}
    \exp(\hat\tau k_{0J_0})_*\pi_*\phi_{0J_0\, *}(\dl_1) 
    \neq \exp(\hat\tau k_{0J_0})_*\pi_*\phi_{0J_0\, *}(\dl_2) -
    \scal{\ud\tau_\nu(\lo)}{\dl_2} (k_\nu - k_{0J_0})(\wh x_{\hat\tau}) .
  \end{equation*}
  {\itshape Claim 2. } If $\scal{\ud\theta^\nu_{01}(\lo)}{\dl_2} < 0 <
  \scal{\ud\theta^\nu_{01}(\lo)}{\dl_1}$ then $L^{\nu
    1} (\dl_1) \ \neq L^{\nu 2}(\dl_2)$, i.e.
  \begin{multline*}
    \exp(\hat\tau k_{0J_0})_*\pi_*\phi_{0J_0\, *}(\dl_1) -
    \scal{\ud\tau_\nu(\lo)}{\dl_1}(k_\nu - k_{0J_0})(\wh
    x_{\hat\tau})\neq \\
    \neq \exp(\hat\tau k_{0J_0})_*\pi_*\phi_{0J_0\, *}(\dl_2) -
    \scal{\ud\tau_\nu(\lo)}{\dl_2} (k_\nu - k_{0J_0})(\wh
    x_{\hat\tau}) - \\
    - \scal{\ud\theta^\nu_{10}(\lo)}{\dl_2}(k_{10} - k_\nu)(\wh
    x_{\hat\tau})
  \end{multline*}
  {\em Proof of Claim 1. } Fix $\nu \in \{1, 2 \}$ and assume, by
  contradiction, that there exist $\dl_1$, $\dl_2 \in T_{\lo}\Lambda$
  such that $\scal{\ud\tau_\nu(\lo)}{\dl_2} < 0 <
  \scal{\ud\tau_\nu(\lo)}{\dl_1}$ and
  \begin{multline}\label{eq:contr1}
    \exp(\hat\tau k_{0J_0})_*\pi_*\phi_{0J_0\, *}(\dl_1) = \\
    = \exp(\hat\tau k_{0J_0})_*\pi_*\phi_{0J_0\, *}(\dl_2) -
    \scal{\ud\tau_\nu(\lo)}{\dl_2}(k_\nu - k_{0J_0})(\wh
    x_{\hat\tau}).
  \end{multline}
  Let $\dx_i := \pi_*\dl_i$, $i=1,2$. Taking the pull-back along the
  reference flow $\wh S_{\wh\tau \, *}$ and using formula
  \eqref{eq:linphi0jind}, equation \eqref{eq:contr1} can be
  equivalently written as
  \begin{multline*}
    \dx_1 - \dx_2 + \sum_{s=0}^{J_0 - 1} \scal{\ud\, (\theta_{0, s+1}
      - \theta_{0s})(\lo)}{\dl_1 - \dl_2} g_{0s}(\wh x_0 )
    + \\
    + \left( - \scal{\ud\tau_\nu(\lo)}{\dl_2} -
      \scal{\ud\theta_{0J_0}(\lo)}{\dl_1 - \dl_2} \right) g_{0J_0}(\wh
    x_0) + \scal{\ud\tau_\nu(\lo}{\dl_2} h_\nu(\wh x_0) = 0 .
  \end{multline*}
  That is, if we define $\dx := \dx_1 - \dx_2$,
  \[
  a_{0s} := \begin{cases}
    \scal{\ud\,(\theta_{0, s+1} - \theta_{0s})(\lo)}{\dl_1 - \dl_2} \quad & s=0, \ldots, J_0 - 1 \\
    - \scal{\ud\theta_{0J_0}(\lo)}{\dl_1 - \dl_2} -
    \scal{\ud\tau_\nu(\lo)}{\dl_2 } \quad & s = J_0
  \end{cases}
  \]
  $b := \scal{\ud\tau_\nu(\lo}{\dl_2}$, and $a_{1j} := 0$ for any
  $j=0, \ldots, J_1$, then $\de := (\dx, a, b) \in V_{10} \cap
  V_{0J_0}^\perp$, so that by \eqref{eq:varsecv10perp}
  \begin{multline*}
    - \scal{\ud\tau_\nu(\lo)}{\dl_2} \bsi \big( \ud\a_*\dx +
    \sum_{s=0}^{J_0 - 1}
    \scal{\ud\, (\theta_{0, s+1} - \theta_{0s})(\lo)}{\ud\a_*\dx} \vG_{0s}(\lo)  \\
    + ( \scal{\ud\theta_{0J_0}(\lo)}{\ud\a_*\dx} -
    \scal{\ud\tau_\nu(\lo)}{\ud\a_*\dx_2}\vG_{0J_0}(\lo) \, , \,
    (\vH_\nu - \vG_{0J_0})(\lo) \big) > 0
  \end{multline*}
  or, equivalently,
  \begin{multline*}
    - \scal{\ud\tau_\nu(\lo)}{\dl_2} \bsi\Big({ \Delta_{0J_0}
      \ud\a_*\dx -
      \scal{\ud\tau_\nu(\lo)}{\ud\a_*\dx_2}\vG_{0J_0}(\lo)}
    \, , \, \\
    { (\vH_\nu - \vG_{0J_0})(\lo)} \Big)> 0 .
  \end{multline*}
  Applying formula \eqref{eq:lintaunuind} we finally get
  \[
  \scal{\ud\tau_\nu(\lo)}{\dl_2} \scal{\ud\tau_\nu(\lo)}{\dl_1}
  \dueforma{\vG_{0J_0}}{\vH_\nu}(\lo) > 0,
  \]
  a contradiction.

  {\em Proof of Claim 2. } Let us fix $\nu \in \{1, 2 \}$ and assume,
  by contradiction, that there exist $\dl_1$, $\dl_2 \in
  T_{\lo}\Lambda$ such that $\scal{\ud\theta_{10}^\nu(\lo)}{\dl_2} < 0
  < \scal{\ud\theta_{10}^\nu(\lo)}{\dl_1}$ and
  \begin{multline}\label{eq:contr2}
    \exp(\hat\tau k_{0J_0})_*\pi_*\phi_{0J_0\, *}(\dl_1) - \scal{\ud\tau_\nu(\lo)}{\dl_1}(k_\nu - k_{0J_0})(\wh x_{\hat\tau}) = \\
    = \exp(\hat\tau k_{0J_0})_*\pi_*\phi_{0J_0\, *}(\dl_2) -
    \scal{\ud\tau_\nu(\lo)}{\dl_2}(k_\nu - k_{0J_0})(\wh x_{\hat\tau})
    - \\
    - \scal{\ud\theta^\nu_{10}(\lo)}{\dl_2}(k_{10} - k_\nu)(\wh
    x_{\hat\tau})
  \end{multline}
  Let $\dx_i := \pi_* \dl_i$, $i=1,2$. Taking the pull-back along the
  reference flow and using formula \eqref{eq:linphi0jind}, equation
  \eqref{eq:contr2} can be equivalently written as
  \begin{multline*}
    \dx_1 - \dx_2 + \! \sum_{s=0}^{J_0 - 1} \!\! \scal{\ud \,
      (\theta_{0,
        s+1} - \theta_{0s})(\lo)}{ \dl_1 - \dl_2}  g_{0s}(\wh x_0) + \\
    + \scal{\ud\, (\tau_\nu - \theta_{0J_0})(\lo)}{\dl_1 - \dl_2}
    g_{0J_0}(\wh x_0)
    + \\
    + \left( - \scal{\ud\tau_\nu(\lo)}{\dl_1 - \dl_2} -
      \scal{\ud\theta_{10}^\nu(\lo)}{\dl_2} \right) h_1 (\wh x_0) +
    \scal{\ud\theta_{10}^\nu(\lo)}{\dl_2} g_{10}(\wh x_0) = 0 .
  \end{multline*}
  That is, if we define $\dx := \dx_1 - \dx_2$,
  \[
  a_{0s} := \begin{cases}
    \scal{\ud\,(\theta_{0, s+1} - \theta_{0s})(\lo)}{\dl_1 - \dl_2}  \quad & s=0, \ldots, J_0 - 1 \\
    \scal{\ud\, (\tau_\nu - \theta_{0J_0})(\lo)}{\dl_1 - \dl_2} \quad
    & s = J_0
  \end{cases}
  \]
  $b := - \scal{\ud\tau_\nu(\lo)}{\dl_1 - \dl_2} -
  \scal{\ud\theta_{10}^\nu(\lo)}{\dl_2}$, and
  \[
  a_{1s} := \begin{cases}
    \scal{\ud \theta^\nu_{10}(\lo)}{\dl_2} \quad & s = 0 \\
    0 \quad & s = 1, \ldots, J_1,
  \end{cases}
  \]
  then $\de := (\dx, a, b) \in V_{10} \cap V_{0J_0}^\perp$ so that
  by Lemma \ref{lem:v10perp},
  \begin{multline*}
    \left( \scal{\ud\tau_\nu(\lo)}{\ud\a_*\dx} +
      \scal{\ud\theta^\nu_{10}(\lo)}{\dl_2} \right)
    \bsi \Big( \ud\a_*\dx + \\
    + \sum_{s=0}^{J_0 - 1}
    \scal{\ud \, (\theta_{0, s+1} - \theta_{0s})(\lo)}{\ud\a_*\dx} \vG_{0s}(\lo) + \\
    + \scal{\ud\, (\tau_\nu - \theta_{0J_0})(\lo)}{\ud\a_*\dx}
    \vG_{0J_0}(\lo) \, , (\vH_\nu - \vG_{0J_0})(\lo) \Big) - \\ -
    \scal{\ud\theta^\nu_{10}(\lo)}{\dl_2} \bsi \Big( \ud\a_*\dx +
    \sum_{s=0}^{J_0
      - 1} \scal{\ud \, (\theta_{0, s+1} - \theta_{0s})(\lo)}{\ud\a_*\dx} \vG_{0s}(\lo) + \\
    + \scal{\ud\, (\tau_\nu - \theta_{0J_0})(\lo)}{\ud\a_*\dx}
    \vG_{0J_0}(\lo) - ( \scal{\ud\tau_\nu(\lo)}{\ud\a_*\dx} +
    \scal{\ud\theta^\nu_{10}(\lo)}{\dl_2}
    )\vH_\nu(\lo)     \, ,  \\
    \, ( \vG_{10} - \vH_\nu )(\lo) \Big) > 0
  \end{multline*}
  or, equivalently,
  \begin{multline*}
    \left( \scal{\ud\tau_\nu(\lo)}{\ud\a_*\dx} +
      \scal{\ud\theta^\nu_{10}(\lo)}{\dl_2} \right) \bsi \Big(
    \Delta_{0J_0} \ud\a_*\dx
    + \scal{\ud\tau_\nu (\lo)}{\ud\a_*\dx} \vG_{0J_0}(\lo) \, , \\
    (\vH_\nu - \vG_{0J_0})(\lo) \Big)
    - \scal{\ud\theta^\nu_{10}(\lo)}{\dl_2} \bsi \Big( \Delta_{0J_0}\ud\a_*\dx + \\
    + \scal{\ud \tau_\nu(\lo)}{\ud\a_*\dx} \vG_{0J_0}(\lo) - (
    \scal{\ud\tau_\nu(\lo)}{\ud\a_*\dx}  \\
    + \scal{\ud\theta^\nu_{10}(\lo)}{\dl_2} )\vH_\nu(\lo) \, , \, (
    \vG_{10} - \vH_\nu )(\lo) \Big) > 0
  \end{multline*}
  that is
  \begin{multline}\label{eq:contr22}
    \Big( \scal{\ud\tau_\nu(\lo)}{\ud\a_*\dx} +
    \scal{\ud\theta^\nu_{10}(\lo)}{\dl_2} \Big) \, \Big( -
    \scal{\ud\tau_\nu(\lo)}{\ud\a_*\dx} +
    \scal{\ud\tau_\nu(\lo)}{\ud\a_*\dx} \Big) \\
    \bsi\big( \vG_{0J_0} \, , \, \vH_\nu)(\lo)   - \\
    - \scal{\ud\theta^\nu_{10}(\lo)}{\dl_2} \Big( -
    \scal{\ud\theta^\nu_{10}(\lo)}{\ud\a_*\dx} -
    \scal{\ud\theta^\nu_{10}(\lo)}{\dl_2} \Big) \bsi\big( \vH_\nu \, ,
    \, \vG_{10}\big)(\lo) > 0.
  \end{multline}
  Since $\ud\a_*\dx = \dl_1 - \dl_2$, we get
  $\scal{\ud\theta^\nu_{10}(\lo)}{\dl_2}
  \scal{\ud\theta^\nu_{10}(\lo)}{\dl_1} \bsi\big( \vH_\nu \, , \,
  \vG_{10}\big)(\lo)> 0$, a contradiction.
\end{proof}

We can now complete the proof of the local invertibility of
$\pi\cH_{\hat\tau}$.  Let us first consider the generic case when
$\ud\,(\tau_1 - \tau_2)(\lo)$ is not identically zero on
$T_\lo\Lambda$.

We need to express the boundaries between the adjacent sectors $M^0$,
$M^{\nu j}$.
\begin{itemize}
\item The boundary between $M^0$ and $M^{11}$ is given by
  \[ \{\dl \in T_\lo\Lambda \colon 0= \scal{\ud\tau_1(\lo)}{\dl}\leq
  \scal{\ud\tau_2(\lo)}{\dl} \}; \]
\item The boundary between $M^0$ and $M^{21}$ is given by
  \[ \{ \dl \in T_\lo\Lambda \colon 0= \scal{\ud\tau_2(\lo)}{\dl}\leq
  \scal{\ud\tau_1(\lo)}{\dl} \}; \]
\item The boundary between $M^{11}$ and $M^{12}$ is given by
  \[ \{\dl \in T_\lo\Lambda \colon
  \scal{\ud\theta^1_{10}(\lo)}{\dl}=0, \
  \scal{\ud\tau_1(\lo)}{\dl}\leq \scal{\ud\tau_2(\lo)}{\dl}\} ; \]
\item The boundary between $M^{21}$ and $M^{22}$ is given by
  \[ \{\dl \in T_\lo\Lambda \colon \scal{\ud\theta^2_{1 0}(\lo)}{\dl}
  =0, \scal{\ud\tau_2(\lo)}{\dl}\leq \scal{\ud\tau_1(\lo)}{\dl} \}; \]
\item The boundary between $M^{12}$ and $M^{22}$ is given by
  \[\{\dl \in T_\lo\Lambda \colon \scal{\ud\tau_2(\lo)}{\dl} =
  \scal{\ud\tau_1(\lo)}{\dl} \leq 0 \}; \]
\end{itemize}
According to Theorem \ref{thm:topo}, in order to prove the
invertibility of our map it is sufficient to prove that both the map
and its linearization are continuous in a neighborhood of $\wh\ell_0$
and of $0$ respectively, that they maintain the orientation and that
there exists a point $\dyb$ whose preimage is a singleton that belongs
to at most two of the above defined sectors.

Notice that the continuity of $\pi\cH_{\hat\tau}$ follows from the
very definition of the maximized flow.  Discontinuities of
$(\pi\cH_{\hat\tau})_*$ may occur only at the boundaries described
above. A direct computation in formulas \eqref{eq:lintau} shows that
this is not the case.  Let us now prove the last assertion.

For ``symmetry'' reasons it is convenient to look for $\dyb$ among
those which belong to the image of the set $\{\dl\in T_\lo\Lambda:0<
\scal{\ud\tau_1(\lo)}{\dl}= \scal{\ud\tau_2(\lo)}{\dl}\}$. Observe
that $ \scal{\ud\tau_1(\lo)}{\dl}= \scal{\ud\tau_2(\lo)}{\dl}$ also
implies $ \scal{\ud\theta^\nu_{10}(\lo)}{\dl} =
\scal{\ud\tau_\nu(\lo)}{\dl}$, $\nu=1,2$, see formulas
\eqref{eq:tautheta}.

Let $\dlb\in T_\lo\Lambda$ such that $0 < \scal{\ud\tau_1(\lo)}{\dlb}
= \scal{\ud\tau_2(\lo)}{\dlb}$ and let $\dyb := L^0 \dlb$.

Clearly $\dyb$ has at most one preimage per each of the above
polyhedral cones. Let us prove that actually its preimage is the
singleton $\{\dlb\}$.

In fact we show that for $\nu, j = 1,2$, there is no $\dl \in M^{\nu
  j}$ such that $L^{\nu j}(\dl) = \dyb$.

{\em 1. } Fix $\nu \in \{1, 2\}$ and assume, by contradiction, that
there exists $\dl \in M^{1 \nu}$ such that $L^{\nu 1}\dl = \dyb$. The
contradiction is shown exactly as in the proof of Claim 1 in Lemma
\ref{lem:equiori}.

{\em 2. }  Fix $\nu \in \{1, 2\}$ and assume, by contradiction, that
there exists $\dl \in M^{\nu 2}$ such that $L^{\nu 2}\dl = \dyb$ that
is: let $\dxb:= \pi_*\dlb$, and $\dx := \pi_*\dl$. Taking the
pull-back along the reference flow at time $\wh\tau$, and recalling
formula \eqref{eq:linphi0jind} we assume by contradiction that
\begin{multline*}
  \dxb - \sum_{s=1}^{J_0} \scal{\ud\theta_{0 s}(\lo)}{\dlb} (g_{0s} -
  g_{0, s-1})(\wh x_0) = \dx - \sum_{s=1}^{J_0} \scal{\ud\theta_{0
      s}(\lo)}{\dl}
  (g_{0,s} - g_{0, s-1})(\wh x_0) - \\
  - \scal{\ud\tau_1(\lo)}{\dl}(h_\nu - g_{0J_0})(\wh x_0) -
  \scal{\ud\theta^\nu_{10}(\lo)}{\dl}(g_{10} - h_\nu)(\wh x_0).
\end{multline*}
or, equivalently,
\begin{multline*}
  \dxb -\dx
  +\sum_{s=1}^{J_0-1}\scal{\ud\,(\theta_{0,s+1}-\theta_{0s})(\lo)}
  { \dlb - \dl})g_{0s}(\wh x_0) - \\
  -\Big(\scal{\ud\theta_{0J_0}(\lo)}{\dlb - \dl} +
  \scal{\ud\tau_1(\lo)}{\dl}\Big)g_{0J_0}(\wh x_0) - \\
  -\scal{\ud\,(\theta^\nu_{10}-\tau_\nu)(\lo)}{\dl} h_1(\wh x_0) +
  \scal{\ud\theta_{10}(\lo)}{\dl}g_{10}(\wh x_0) = 0.
\end{multline*}
Let $\de := (\dxb - \dx,a,b)$, where,
\[
\begin{split}
  & a_{0s} :=
  \begin{cases}
    \scal{\ud\,(\theta_{0,s+1}-\theta_{0s})(\lo)}{\dlb - \dl}
    \quad & s = 0, \ldots, J_0 -1 ,\\
    \scal{\ud\theta_{0J_0}(\lo)}{\dlb - \dl} -
    \scal{\ud\tau_1(\lo)}{\dl} \quad & s = J_0 ,
  \end{cases} \\
  & b := -\scal{\ud\,(\theta^\nu_{10}-\tau_\nu)(\lo)}{\dl}  , \\
  & a_{1s} :=
  \begin{cases}
    \scal{\ud\theta^\nu_{10}(\lo)}{\dl}  \quad & s = 0,\\
    a_{1s} = 0 \quad & s =1, \ldots, J_1 .
  \end{cases}
\end{split}
\]
Then $\de \in V_{10} \cap V_{0J_0}^{\perp_{J\se_\nu}}$ and Lemma
\ref{lem:v10perp} applies:
\begin{align*}
  0 < & J\se_\nu[\de]^2 = -b \, \dueforma{\dlb - \dl + \sum_{s=0}^{J_0} a_{0s} \vG_{0s}(\lo)}{(\vH_\nu - \vG_{0J_0})(\lo)} - \\
  & - a_{10}\, \dueforma{ \dlb - \dl
    +\sum_{s=0}^{J_0}a_{0s}\vG_{0s}(\lo)+b\vH_\nu(\lo)}
  {(\vG_{10}-\vH_1)(\lo)} = \\
  = & \scal{\ud\,(\theta^\nu_{10}-\tau_\nu)(\lo)}{\dl} \left(
    \scal{\ud\tau_\nu(\lo)}{\dlb - \dl} - \scal{\ud\tau_\nu(\lo)}{\dl}
  \right) \dueforma{\vG_{0J_0}}{\vH_\nu}(\lo) - \\ %
  & - \scal{\ud\theta^\nu_{10}(\lo)}{\dl} \Big( \big( -
  \scal{\ud\theta^\nu_{10}(\lo)}{\dlb - \dl} -
  \scal{\ud\theta^\nu_{10}(\lo)}{\dl}\big)
  \dueforma{\vH_\nu}{\vG_{10}}(\lo) + \\
  & + \scal{\ud\tau_\nu(\lo)}{\dlb}\dueforma{\vG_{0J_0}}{\vH_{3 - \nu}}(\lo) \Big) = \\
  = & \scal{\ud\,(\theta^\nu_{10}-\tau_\nu)(\lo)}{\dl}
  \scal{\ud\tau_\nu(\lo)}{\dlb} \dueforma{\vG_{0J_0}}{\vH_\nu}(\lo) - \\
  & - \scal{\ud\theta^\nu_{10}(\lo)}{\dl} \Big(
  \scal{\ud\theta^\nu_{10}(\lo)}{\dlb}
  \dueforma{\vH_\nu}{\vG_{10}}(\lo) + \\
  & + \scal{\ud\tau_\nu(\lo)}{\dlb} \dueforma{\vG_{0J_0}}{\vH_{3 -
      \nu}}(\lo) \Big)
\end{align*}
which is a contradiction, since all the addenda are negative.
 
By Theorem \ref{thm:topo} this proves the invertibility of
$\pi\cH_{\hat\tau}$, hence $\psi$ is one-to-one in a neighborhood of
$[0, \hat\theta_{10} - \ep] \times \big\{\lo\big\}$.


Assume now that the non generic case $T_{\lo}\Lambda \subset \ker
\ud(\tau_1 - \tau_2)(\lo)$ holds.
We are going to prove the Lipschitz invertibility of
$\left.\pi\cH_{\hat\tau}\right\vert_\Lambda$ by means of Clarke's
inverse functions theorem, see \cite{Cla83}.  The generalized Jacobian
$\partial (\pi\cH_{\hat\tau})(\lo)$ (in the sense of Clarke) of
$\pi\cH_{\hat\tau} \colon \Lambda \to M$ at $\lo$ is the closed convex
hull of the linear maps $L^0$, $L^{\nu j}$, $\nu, j =1, 2$ defined in
\eqref{eq:lintau}.

We distinguish between two sub-cases: \\
{\em 1. } $ \scal{\ud\tau_1(\lo)}{\dl} =  \scal{\ud\tau_2(\lo)}{\dl} = 0 \text{ for any }\dl \in T_\lo \Lambda$ \\
In this case we also have $\ud\theta^1_{10}(\lo)\vert_{T_\lo \Lambda}
\equiv \ud\theta^2_{10}(\lo)\vert_{T_\lo \Lambda} \equiv 0$, see
formulas \eqref{eq:tautheta}, hence all the linear maps $L^0$, $L^{\nu
  j}$, $\nu, j =1, 2$ defined in \eqref{eq:lintau} coincide with the
map $L^0$, so that $\pi\cH_{\hat\tau}$ is differentiable at $\lo$. The
invertibility of
$L^0$ and Clarke's invertibility theorem yield the claim. \\
{\em 2. } $ \scal{\ud\tau_1(\lo)}{\dl} = \scal{\ud\tau_2(\lo)}{\dl}
\text{ for any }\dl \in T_\lo\Lambda$ but
$\ker(\ud\tau_1(\lo)\vert_{T_\lo \Lambda}) \neq T_\lo \Lambda$.  In
this case we also have $\ud\theta^1_{10}(\lo)\vert_{T_\lo \Lambda}
\equiv \ud\theta^2_{10}(\lo)\vert_{T_\lo \Lambda} \equiv
\ud\tau_1(\lo)\vert_{T_\lo \Lambda}$ (see formulas
\eqref{eq:tautheta}) so that $L^{12} \equiv L^{22}$.

Let $\{v_1, v_2, \ldots, v_n \}$ be a basis of $T_{\wh x_0}M$ such
that $\scal{\ud\tau_1(\lo)}{\ud\a_* v_1} = 1$ and
$\scal{\ud\tau_1(\lo)}{\ud\a_* v_i} = 0$ for $i=2, \ldots, n$.  We
will show that $\partial(\pi\cH_{\hat\tau})(\lo)$ is made up of
invertible matrices by showing that
\[
(L^0)^{-1}\big(t_0 L^0 + t_1 L^{11} + t_2 L^{21} + t_3 L^{12} + t_4
L^{22} \big)\circ\ud\a_*
\]
is invertible for any $t_0, \ldots, t_4 \geq 0$ such that
$\sum_{i=0}^4 t_i = 1$.

Let $c^\nu_i$, $\nu=1, 2$, $i=1, \ldots, n$ such that
\[
(h_\nu - g_{0J_0})(\wh x_0) = \sum_{i=1}^n c_i^\nu v_i.
\]
We have
\begin{equation*}
  (L^0)^{-1} L^{\nu j} \ud\a_* v_i = v_i \quad 
  i=2, \ldots, n \text { and } \nu, \, j = 1, 2  
\end{equation*}
and, for each $\nu=1, 2$:
\begin{equation*}
  \begin{split}
    (L^0)^{-1}L^{\nu 1} \ud\a_* v_1 & = v_1 - (h_\nu - g_{0J_0})(\wh
    x_0) =
    (1 - c^\nu_1)v_1 - \sum_{k=2}^n c^\nu_k v_k \\
    (L^0)^{-1} L^{\nu 2} \ud\a_* v_1 & = v_1 - (h_\nu - g_{0J_0})(\wh
    x_0) -
    (g_{10} - h_\nu )(\wh x_0) = \\
    & = (1 - c^1_1 - c^2_1)v_1 - \sum_{k=2}^n (c^1_k + c^2_k) v_k.
  \end{split}
\end{equation*}
Thus the determinant of $ (L^0)^{-1}\big(t_0 L^0 + t_1 L^{11} + t_2
L^{21} + t_3 L^{12} + t_4 L^{22} \big)\circ\ud\a_*$ is given by $t_0 +
t_1 \det (L^0)^{-1}L^{1 1} \ud\a_* + t_2 \det (L^0)^{-1}L^{2 1}\ud\a_*
+ (t_3 + t_4) \det (L^0)^{-1}L^{12} \ud\a_*$ which cannot be null
since all the addenda are positive as it follows from Lemmata
\ref{lem:equiori} and \ref{lem:lemalg}.  This concludes the proof of
the invertibility of $\pi\cH_{\hat\tau}$. Let us now turn to
$\pi\cH_{\hat\theta_{1j}}$, $j=1, \ldots, J_1$.

 
 

\medskip

For any $j=1, \ldots , J_1$, there are four regions in $\Lambda$,
characterized by the following properties
\begin{align*}
  &\{\ell\in \Lambda \colon \theta_{1j}(\ell) \geq \hat\theta_{1j}
  \text{ and }
  \theta_{0, J_0+1}(\ell) = \tau_1(\ell)\},\\
  &\{\ell\in \Lambda \colon \theta_{1j}(\ell) \geq
  \hat\theta_{1j}\text{ and }
  \theta_{0, J_0+1}(\ell) = \tau_2(\ell)\},\\
  &\{\ell\in \Lambda \colon \theta_{1j}(\ell) < \hat\theta_{1j} \text{
    and }
  \theta_{0, J_0+1}(\ell) = \tau_1(\ell)\}, \\
  &\{\ell\in \Lambda \colon \theta_{1j}(\ell) < \hat\theta_{1j}\text{
    and } \theta_{0, J_0+1}(\ell) = \tau_2(\ell)\}.
\end{align*}
As for $\pi\cH_{\hat\tau}$, $\pi\cH_{\hat\theta_{1j}}$ turns out to be
a Lipschitz continuous, piecewise $C^1$ application. Its invertibility
can be proved applying again Theorem \ref{thm:topo}. Let us write the
piecewise linearized map $(\pi\cH_{\hat\theta_{1j}})_*$
\begin{itemize}
\item Let $N_{1j}^{10} := \{ \dl \in T_{\lo}\Lambda \colon
  \scal{\ud\tau_1(\lo)}{\dl}\leq \scal{\ud\tau_2(\lo)}{\dl}, \
  \scal{\ud\theta^1_{1j}(\lo)}{\dl} \geq 0 \} $. Then
  \[
  (\pi\cH_{\hat\theta_{1j}})_*\dl = A_{1j}^{1}\dl :=
  \exp(\hat\theta_{1j}k_{1, j-1\, *})\pi_*\varphi^1_{1, j-1 \, *}(\dl)
  \]
\item Let $N_{1j}^{20} := \{\dl \in T_{\lo}\Lambda \colon
  \scal{\ud\tau_2(\lo)}{\dl}\leq \scal{\ud\tau_1(\lo)}{\dl}, \
  \scal{\ud\theta^2_{1j}(\lo)}{\dl} \geq 0 \} $. Then
  \[
  (\pi\cH_{\hat\theta_{1j}})_*\dl = A_{1j}^{2}\dl :=
  \exp(\hat\theta_{1j} k_{1, j-1\, *})\pi_*\varphi^2_{1, j-1 \,
    *}(\dl)
  \]
\item Let $N_{1j}^{11} := \{ \dl \in T_{\lo}\Lambda \colon
  \scal{\ud\tau_1(\lo)}{\dl}\leq \scal{\ud\tau_2(\lo)}{\dl}, \
  \scal{\ud\theta^1_{1j}(\lo}{\dl} \leq 0 \} $. Then
  \begin{multline*}
    (\pi\cH_{\hat\theta_{1j}})_*\dl = B_{1j}^{1}\dl :=
    \exp(\hat\theta_{1j} k_{1, j-1\, *}) \pi_*\varphi^1_{1, j-1 \,
      *}(\dl) - \\
    - \scal{\ud\theta^1_{1j}(\lo)}{\dl} (k_{1j} - k_{1, j-1})(\wh
    x_{1j})
  \end{multline*}
\item Let $N_{1j}^{21} := \{ \dl \in T_{\lo}\Lambda \colon
  \scal{\ud\tau_2(\lo)}{\dl}\leq \scal{\ud\tau_1(\lo)}{\dl}, \
  \scal{\ud\theta^2_{1j}(\lo)}{\dl}\leq 0 \} $. Then
  \begin{multline*}
    (\pi\cH_{\hat\theta_{1j}})_*\dl = B_{1j}^{2}\dl :=
    \exp(\hat\theta_{1j}\vK_{1, j-1\, *}) \pi_* \varphi^2_{1, j-1 \,
      *}(\dl) - \\
    - \scal{\ud\theta^2_{1j}(\lo)}{\dl}(k_{1j} - k_{1, j-1})(\wh
    x_{1j})
  \end{multline*}
\end{itemize}
Analogously to what we did at time $\wh\tau$, let us first consider
the non degenerate case $\scal{\ud\, (\tau_1 - \tau_2)(\lo)}{\dl} \neq
0$ for some $\dl \in T_{\lo}\Lambda$: according to Theorem
\ref{thm:topo}, we only have to prove that both the map and its
piecewise linearization are continuous in a neighborhood of
$\wh\ell_0$ and of $0$ respectively, that the linearized pieces are
orientation preserving and that there exists a point $\dyb$ whose
preimage is a singleton.

The only nontrivial part is the last statement which can be proved by
picking $\dyb \in A_{1j}^1(N_{1j}^{10})\cap A_{1j}^2(N_{1j}^{20})$:
let $\dlb\in T_\lo\Lambda$ such that $ \scal{\ud\tau_1(\lo)}{\dlb} =
\scal{\ud\tau_2(\lo)}{\dlb} > 0$ and let $\dyb := A_{1j}^1 \dlb =
A_{1j}^2\dlb$.

Let $\nu \in \{1, 2\}$ and assume, by contradiction, that there exists
$\dl_\nu \in N_{1j}^{\nu1}$ such that $B^\nu_{1j}\dl_1 = \dyb$, i.e.
\begin{multline*}
  \exp(\hat\theta_{1j}k_{1, j-1\, *})\pi_*\varphi^\nu_{1, j-1 \, *}(\dlb) = \\
  = \exp(\hat\theta_{1j} k_{1, j-1\, *}) \pi_*\varphi^\nu_{1, j-1 \,
    *}(\dlb ) - \scal{\ud\theta^\nu_{1j}(\lo)}{\dlb} ( k_{1j} - k_{1,
    j-1})(\wh x_{1j}).
\end{multline*}
Taking the pull-back along the reference flow $\wh
S_{\hat\theta_{1j}}$ and defining $\dxb := \pi_*\dlb $, $\dx_\nu :=
\pi_*\dl_\nu $ we can equivalently write
\begin{multline*}
  \dx - \dxb - \sum_{s=1}^{J_0} \scal{\ud\theta_{0s}(\lo)}{\dl - \dlb}
  (g_{0s} - g_{0, s-1})(\wh x_0)
  - \\
  - \scal{\ud\tau_\nu(\lo)}{\dl - \dlb } (h_\nu - g_{0J_0})(\wh x_0)
  - \scal{\ud\theta^\nu_{10}(\lo)}{\dl - \dlb}(g_{10} - h_\nu)(\wh x_0) - \\
  - \sum_{s=1}^{j-1}\scal{\ud\theta^\nu_{1s}(\lo)}{\dl - \dlb}(g_{1s}
  - g_{1, s-1})(\wh x_0) - \scal{\ud\theta^\nu_{1j}(\lo)}{\lo} (g_{1j}
  - g_{1, j-1})(\wh x_0) = 0
\end{multline*}
that is
\begin{multline*}
  \dxb - \dx + \sum_{s=0}^{J_0 - 1}\scal{\ud\, (\theta_{0,s + 1} -
    \theta_{0s}(\lo)}{\dlb - \dl} g_{0s}(\wh x_0) + \\
  + \scal{\ud\, (\tau_\nu - \theta_{0J_0})(\lo)}{\dlb - \dl} g_{0J_0}(\wh x_0)  + \\
  + \scal{\ud \, ( \theta^\nu_{10} - \tau_\nu)(\lo)}{\dlb - \dl }
  h_\nu(\wh x_0) + \sum_{s=0}^{j-2} \scal{\ud\, ( \theta^\nu_{1, s+1}
    - \theta^\nu_{1s})(\lo)}{\dlb - \dl}
  g_{1s} (\wh x_0) +  \\
  + \left( \scal{\ud\theta^\nu_{1j}(\lo)}{\dl} -
    \scal{\ud\theta^\nu_{1, j-1}(\lo)}{\dlb - \dl} \right)g_{1,
    j-1}(\wh x_0) - \scal{\ud\theta^\nu_{1j}(\lo)}{\dl} g_{1j}(\wh
  x_0) = 0
\end{multline*}
Let $\de := (\dxb - \dx,a,b)$, where,
\[
\begin{split}
  & a_{0s} :=
  \begin{cases}
    \scal{\ud\,(\theta_{0,s+1}-\theta_{0s})(\lo)}{\dlb - \dl} \quad & s = 0, \ldots, J_0 -1 ,\\
    \scal{ \ud\, (\tau_\nu - \theta_{0J_0})(\lo)}{\dlb - \dl} \quad &
    s = J_0 ,
  \end{cases} \\
  & b := \scal{\ud\,(\theta^\nu_{10}-\tau_\nu)(\lo)}{\dlb - \dl} \\
  & a_{1s} :=
  \begin{cases}
    \scal{ \ud\, ( \theta^\nu_{1, s+1} - \theta^\nu_{1s}(\lo)}{\dlb -
      \dl}
    \quad & s = 0, \ldots, j-2 \\
    \scal{\ud\theta^\nu_{1j}(\lo)}{\dl} -  \scal{\ud\theta^\nu_{1, j-1}(\lo)}{\dlb - \dl}\quad &  s = j - 1 ,\\
    -   \scal{\ud\theta^\nu_{1j}(\lo)}{\dl} \quad &  s = j  ,\\
    0 & s= j+1 , \ldots, J_1.
  \end{cases}
\end{split}
\]
Then $\de \in V_{1j} \cap V_{1, j-1}^{\perp_{J\se_\nu}}$ and Lemma
\ref{lem:v1jperp} applies:
\begin{align*}
  0 > & a_{1j} \, {\bsi}\Big( \ud\a_* (\dx - \dxb )
  + \sum_{s=0}^{J_0}a_{0s}\vG_{0s}(\lo) + \\
  & \qquad\qquad + b\vH_\nu(\lo) + \sum_{s=0}^{j-1}a_{1s}\vG_{1s}(\lo)
  \, , \, (\vG_{1j}-\vG_{1, j-1})(\lo)
  \Big) = \displaybreak[0]\\
  = & \scal{\ud\theta^\nu_{1j}(\lo)}{\dl}\Big\{
  \scal{\ud\theta^\nu_{1j}(\lo)}{\dl - \dlb}
  \dueforma{\vG_{1, j-1}}{\vG_{1j}}(\lo) - \\
  & - \scal{\ud\theta^\nu_{1j}(\lo)}{\dl} \dueforma{\vG_{1,
      j-1}}{\vG_{1j}}(\lo)
  \Big\} = \displaybreak[0]\\
  = & - \scal{\ud\theta^\nu_{1j}(\lo)}{\dl}
  \scal{\ud\theta^\nu_{1j}(\lo)}{\dlb} \dueforma{\vG_{1,
      j-1}}{\vG_{1j}}(\lo),
\end{align*}
a contradiction.
  
Let us now turn to the degenerate case
$\ud\tau_1\vert_{T_{\lo}\Lambda} \equiv
\ud\tau_2\vert_{T_{\lo}\Lambda}$.  From equations
\eqref{eq:lintheta1jind} one can recursively show that $
\scal{\ud\theta^1_{1j}(\lo)}{\dl} \vert_{T_{\lo}\Lambda} =
\scal{\ud\theta^2_{1j}(\lo)}{\dl}\vert_{T_{\lo}\Lambda}$ for any $\dl
\in T_{\lo}\Lambda$ and for any $j=1, \ldots, J_1$, so that $A^1_{1j}
= A^2_{1j}$ and $B^1_{1j} = B^2_{1j}$ and the result can be proved
repeating the proof of Lemma \ref{lem:maxflow0j}.

This proves the invertibility of $\pi\cH_{\hat\theta_{1j}}$, $j=1,
\ldots, J_1$. Thus the map
\[
\id\times \pi\cH\colon [0, T] \times \Lambda \to M
\]
is one-to-one from a neighborhood of $[0, T] \times
\{\wh\lambda(0)\}$ in $ [0, T] \times \Lambda$ and we can apply the
procedure described in Section \ref{sec:hamimeth}.

\subsection{Proof of Theorem \ref{thm:main}}
\label{sec:proof}
Let
\[
\id\times \pi\cH\colon [0, T] \times \cO \to \cV = [0, T] \times \cU
\]
be one-to-one and let $\xi \colon [0, T] \to M$ be an admissible
trajectory whose graph is in $\cV$.

Applying the Hamiltonian methods explained in Section
\ref{sec:hamimeth} we have:
\[
C(\xi, u) - C(\wh\xi, \wh u) \geq \cF(\xi(T)) - \cF(\wh x_f).
\]
Thus, to complete the proof of Theorem \ref{thm:main} it suffices to
show that $\cF$ has a local minimum at $\wh x_f$.  In order to shorten
the notation, let us denote $\psi_T(\ell) := (\pi\cH_T)^{-1}(\ell)$.
\begin{theorem}
  $F$ has a strict local minimum at $\wh x_f$.
\end{theorem}
\begin{proof}
  It suffices to prove that
  \begin{equation}
    \ud\, \cF (\wh x_f) = 0 \text{ and } 
    \uD^2  \cF (\wh x_f) > 0 \, . \label{eq:D2F2}
  \end{equation}
  The first equality in \eqref{eq:D2F2} is an immediate consequence of
  the definition
  of $\cF$ and of PMP. Let us prove that also the inequality  holds. \\
  Since $\ud \left( \a\circ\pi \psi_T \right) = \cH_T\circ\psi_T$, we
  also have
  \begin{align}
    \ud \, \cF & = \cH_T\circ\psi_T + \ud\beta \\
    \begin{split}
      \uD^2 \cF(\wh x_f) [\dx_f]^2 & = \left( (\cH_T\circ\psi_T)_*
        +\uD^2 \beta \right)(\wh x_f) [\dx_f]^2  \\
      & = \dueforma{(\cH_T\circ\psi_T)_* \dx_f}{\ud \, (-\beta)_*
        \dx_f}\label{eq:deriv}
    \end{split}
  \end{align}
  From Lemma \ref{lem:Njperp} we have
  \begin{multline}\label{eq:finale1}
    \bsi \Big(\ud\,(-\wh\beta)_* \big( \dx +
    \dsum_{i=0}^1\dsum_{s=0}^{J_i}
    a_{is} g_{is}(\wh x_0) + b h_\nu (\wh x_0) \big) \, , \\
    \ud\alpha_* \dx + \dsum_{i=0}^1 \dsum_{s=0}^{J_i} a_{is}
    \vG_{is}(\lo) + b \vH_{\nu}(\lo) \Big) < 0.
  \end{multline}
  Applying $\wh\cH_{T*}$ to both arguments and using the
  anti-simmetry property of $\sigma$ we get
  \[
  \dueforma{\cH_{T*}\ud\alpha_*\dx}
  {\ud\,(-\beta)_*((\pi\cH_T)_*\ud\alpha_*\dx)} > 0
  \]
  which is exactly \eqref{eq:deriv} with $\dx :=
  \pi_*\psi_{T*}\dx_f$.
\end{proof}

To conclude the proof of Theorem \ref{thm:main} we have to show that
$\wh\xi$ is a strict minimizer.  Assume $C(\xi, u) = C(\wh\xi, \wh
u)$.  Since $\wh x_f$ is a strict minimizer for $F$, then $\xi(T) =
\wh x_f$ and equality must hold in \eqref{eq:integ}:
\begin{equation*}
  \scal{\cH_s(\psi^{-1}_s(\xi(s)))}{\dot\xi(s)}
  =  H_s(\cH_s(\psi^{-1}_s(\xi(s)))).
\end{equation*}
By regularity assumption, $u(s) = \wh u(s)$ for any $s$ at least in a
left neighborhood of $T$, hence $\xi(s) = \wh\xi(s)$ and $
\psi^{-1}_s(\xi(s)) = \lo$ for any $s$ in such neighborhood. $u$ takes
the value $\wh u_{\vert (\hat\theta_{1J_1}, T)}$ until $\cH_s
\psi^{-1}_s(\xi(s)) = \cH_s(\lo) =\wh\lambda(s)$ hits the
hyper-surface $K_{1,J_1} = K_{1, J_1 - 1}$, which happens at time $s =
\hat\theta_{1,J_1}$. At such time, again by regularity assumption, $u$
must switch to $\wh u\vert_{(\hat\theta_{1, J_1-1},\hat\theta_{1, J_1}
  )}$, so that $\xi(s) = \wh\xi(s)$ also for $s$ in a left
neighborhood of $\hat\theta_{1, J_1}$.  Proceeding backward in time,
with an induction argument we finally get $(\xi(s), u(s)) =
(\wh\xi(s), \wh u(s))$ for any $s \in [0, T]$.

In the abnormal case the cost is zero, thus the existence of a strict
local minimiser implies that the trajectory is isolated among
admissible ones.

\section{Appendix: Invertibility of piecewise $C^1$ maps}

This Section is devoted to piecewise linear maps and to piecewise
$C^1$ maps. Our aim is to prove a sufficient condition, in terms of the
``piecewise linearization'', of piecewise $C^1$ maps. 

Some  linear algebra preliminaries are needed. 
\begin{lemma}\label{lem:lemalg}
  Let $A$ and $B$ be linear automorphisms of $\R^n$. Assume that for
  some $v\in(\R^n)^* \setminus \{0 \}$, $A$ and $B$ coincide on the
  space $\pi(v) := \{x\in\R^n \colon \scal{v}{x} = 0 \}$.  Then, the map
  $\mathcal{L}_{AB}$ defined by $x\mapsto Ax$ if $ \scal{v}{x} \geq
  0$, and by $x\mapsto Bx$ if $ \scal{v}{x} \leq 0$, is a
  homeomorphism if and only if $\det(A)\cdot\det(B)>0$.
\end{lemma}
\begin{proof}
Let $w_1, \ldots, w_{n-1}$ be a basis of the hyperplane $\pi(v)$. We
complete it with $v$ to obtain a basis of $\R^n$.
The matrix of $A^{-1}B$ in this basis is given by
\[
\left(
\begin{array}[c]{c | c}
  \boldsymbol{I}_{n-1} &
  \begin{matrix}
    \gamma_1 \\
\vdots \\
\gamma_{n-1} 
  \end{matrix} \\[7mm]
\hline \rule{0pt}{4mm} \boldsymbol{0}^t_{n-1} & \gamma_n
\end{array}\right)
\]
where $ \boldsymbol{I}_{n-1}$ is the $n-1$ unit matrix and
$\boldsymbol{0}_{n-1}$ is the $n-1$ null vector and the $\gamma_i$'s 
are defined by
\[
A^{-1}Bv = \dsum_{i=1}^{n-1}\gamma_i w_i + \gamma_n v.
\]
Thus $\gamma_n $ is positive if and only if $\det(A) \det(B)$ is
positive and $\gamma_n$ is zero if and only either $A$ or $B$ is not
invertible.

Observe that if $\gamma_n$ is negative, then 
\[
\cL_{AB}v =
\cL_{AB}
\left( \dsum_{i=1}^{n-1}-\dfrac{\gamma_i}{\gamma_n}\, w_i +
\frac{1}{\gamma_n} \, v
\right).
\] 
Thus, in this case $\cL_{AB}$ is not one--to--one.
 
We now prove that $\cL_{AB}$ is injective if
$\gamma_n$ is positive. Assume this is not true. Since both $A$ and
$B$ are invertible, there exist $z_A, z_B \in\R^n$  such
that $\scal{v}{z_A} > 0$, $\scal{v}{z_B} < 0$ and $A z_A = B z_B$ or,
equivalently,  $A^{-1}B z_B = z_A$. Let
\begin{equation*}
   z_A = \sum_{i=1}^{n-1}c^i_A w_i + c_A v, \qquad  z_B =
   \sum_{i=1}^{n-1}c^i_B w_i + c_B v .
\end{equation*}
Clearly $c_A >0$, $c_B < 0$. The equality  $A^{-1}B z_B = z_A$ is
equivalent to 
\[
\sum_{i=1}^{n-1}c^i_B  w_i + c_B \sum_{i=1}^{n-1}\gamma_i  w_i + c_B
\gamma_n v =  \sum_{i=1}^{n-1}c^i_A  w_i + c_A v. 
\]
Consider the scalar product with $v$, we get
$
c_B \gamma_n \norm{v}^2 = c_A\norm{v}^2
$,
which is a contradiction.

We finally prove that, if $\gamma_n$ is positive, then $\cL_{AB}$ is
surjective.
Let $z \in \R^n$. There exist $y_A$, $y_B \in \R^n$ such that  
$A y_A = B y_B = z$. If either $\scal{v}{y_A} \geq 0$ or
$\scal{v}{y_B} \leq 0$, there is nothing to prove.
Let us assume $\scal{v}{y_A} < 0$ and $\scal{v}{y_B} >  0$.
In this case $A^{-1}B y_B = y_A$ and proceeding as above we get a
contradiction.
\end{proof}

\begin{definition}
  Let $G:\R^n\to\R^n$ be a continuous, piecewise linear map at $0$, in
  the sense that $G$ is continuous and there exists a decomposition
  $S_1,\ldots,S_k$ of $\R^n$ in closed polyhedral cones (intersection
  of half spaces, hence convex) with nonempty interior and common
  vertex in the origin and such that $\partial S_i \cap \partial S_j =
  S_i \cap S_j$, $i \neq j$, and linear maps $L_1,\ldots,L_k$ with
  \[
  G(x) = L_i x, \qquad x \in S_i,
  \]
  with $L_i x = L_j x$ for any $x \in S_i \cap S_j$, and $\det L_i\neq
  0$, $\forall i=1,\ldots,k$.
\end{definition}
  
\begin{example}
As an example of continuous piecewise linear map consider $G\colon \R^2\to \R^2$
given by
 
  \begin{gather*}
    L_1 =
    \begin{pmatrix}
      1 & 0 \\
      0 & 1
    \end{pmatrix}
   \quad L_2 =
    \begin{pmatrix}
      1 & -\sqrt{2} \\
      0 & \sqrt{2} - 1
    \end{pmatrix}
   \quad L_3 =
    \begin{pmatrix}
      -\sqrt{2} & -\sqrt{2}  + 1 \\
      1 & 0
    \end{pmatrix} \\[5mm]
    L_4=
    \begin{pmatrix}
      0 & 1 \\
      -\sqrt{2} + 1 & - \sqrt{2}
    \end{pmatrix}
   \qquad L_5 =
    \begin{pmatrix}
      \sqrt{2} -1 & 0 \\ 
      -\sqrt{2} & 1
    \end{pmatrix}
    \end{gather*}

 where the $L_i$'s are applied in the corresponding cone $S_i$
 illustrated in picture \ref{fig:esempio}

\definecolor{grigio0}{gray}{0.9}
\definecolor{grigio1}{gray}{0.8}
\definecolor{grigio2}{gray}{0.7}
\definecolor{grigio3}{gray}{0.6}
\definecolor{grigio4}{gray}{0.5}

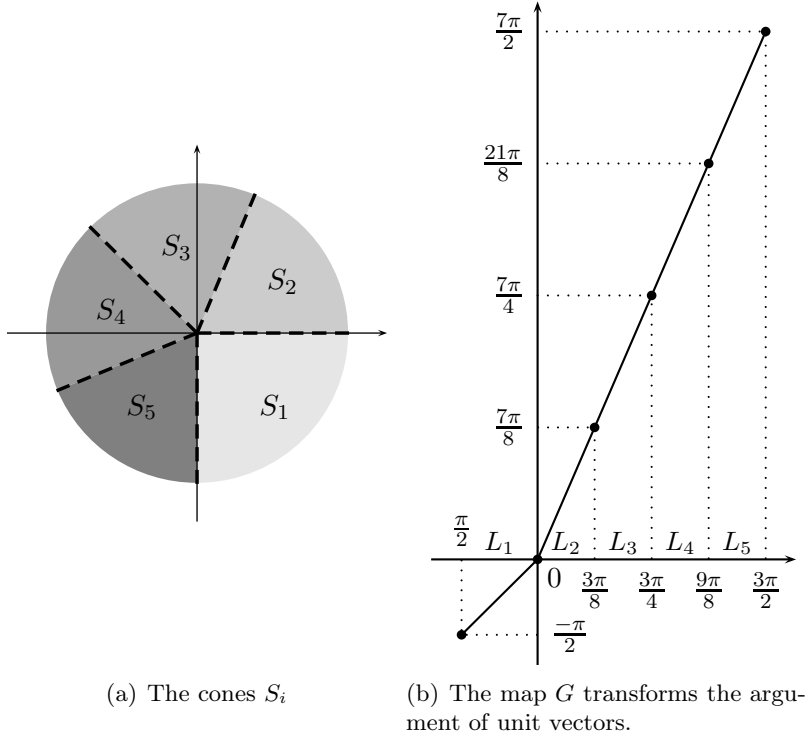
\begin{figure}[ht!]
  \centering
\subfigure[The cones $S_i$]{
    \begin{pspicture}(-2.5, -4.4)(2.5,4.4) 
      \pswedge[linestyle=none,fillstyle=solid, fillcolor=grigio0](0,0){2}{-90}{0}
      \pswedge[linestyle=none,fillstyle=solid, fillcolor=grigio1](0,0){2}{0}{67.5}
      \pswedge[linestyle=none,fillstyle=solid, fillcolor=grigio2](0,0){2}{67.5}{135}
      \pswedge[linestyle=none,fillstyle=solid, fillcolor=grigio3](0,0){2}{135}{202.5}
      \pswedge[linestyle=none,fillstyle=solid, fillcolor=grigio4](0,0){2}{202.5}{270} 
     \psline[linewidth=0.5pt]{->}(-2.5, 0)(2.5, 0) %
     \psline[linewidth=0.5pt]{->}(0,-2.5)(0,2.5)
\SpecialCoor
\rput(0,0){
\psline[linestyle=dashed,linewidth=1.3pt](0,0)(2;67.5)
\psline[linestyle=dashed,linewidth=1.3pt](0,0)(2;135)
\psline[linestyle=dashed,linewidth=1.3pt](0,0)(2;202.5)
\psline[linestyle=dashed,linewidth=1.3pt](0,0)(2;270)
\psline[linestyle=dashed,linewidth=1.3pt](0,0)(2;0)
       \uput[dr](0.7,-0.7){$S_1$}
       \uput[ur](0.8,0.4){$S_2$}
       \uput[u](-0.27,0.8){$S_3$}
       \uput[ul](-0.8,0){$S_4$}
       \uput[dl](-0.4,-0.7){$S_5$}
 }
    \end{pspicture}
}
\subfigure[The map $G$ transforms the argument of unit vectors.]{
  \begin{pspicture}(-1.6, -1.4)(3.4, 7.4)
    \psline{->}(-1.4, 0)(3.4, 0)
    \psline{->}(0, -1.4)(0, 7.4)
    \psline[showpoints=true]{-}(-1, -1)(0,0)(.75, 1.75)(1.5, 3.5)(2.25, 5.25)(3, 7)
    \uput[u](-1,0){$\frac{\pi}{2}$}
    \uput[dr](0,0){$0$}
    \uput[u](-0.5,-0.1){\small$L_1$}
    \uput[u](.375,-0.1){\small$L_2$}
    \uput[d](.75,0){$\frac{3\pi}{8}$}
    \uput[u](1.125,-0.1){\small$L_3$}
    \uput[u](1.875,-0.1){\small$L_4$}
    \uput[u](2.625,-0.1){\small$L_5$}
    \uput[d](1.5,0){$\frac{3\pi}{4}$}
    \uput[d](2.25,0){$\frac{9\pi}{8}$}
    \uput[d](3,0){$\frac{3\pi}{2}$}
    \uput[r](0, -1){$\frac{-\pi}{2}$}
    \uput[l](0, 1.75){$\frac{7\pi}{8}$}
    \uput[l](0, 3.5){$\frac{7\pi}{4}$}
    \uput[l](0, 5.25){$\frac{21\pi}{8}$}
    \uput[l](0, 7){$\frac{7\pi}{2}$}
    \psline[linestyle=dotted]{-}(-1, 0)(-1, -1)(0, -1)
    \psline[linestyle=dotted]{-}(.75, 0)(.75, 1.75)(0, 1.75)
    \psline[linestyle=dotted]{-}(1.5, 0)(1.5, 3.5)(0, 3.5)
    \psline[linestyle=dotted]{-}(2.25, 0)(2.25, 5.25)(0, 5.25)
    \psline[linestyle=dotted]{-}(3, 0)(3, 7)(0, 7)
\end{pspicture}
}
  \caption{Polyhedral cones and the transformation of unit vectors
    under $G$}
  \label{fig:esempio}
\end{figure}
\end{example}

Observe that any continuous piecewise linear map $G$ is differentiable in
$\R^n\setminus\cup_{i=1}^k\partial S_i$.  It is easily shown that $G$
is proper, and therefore $\deg(G,\R^n,p)$ is well-defined for any
$p\in\R^n$ (the construction in \cite{Mil65}, Chapter 5 is still valid if the
assumption on the compactness of the manifolds is replaced with the
assumption that $G$ is proper. Compare also \cite{BFPS03}).  Moreover $\deg(G,\R^n,p)$ is constant
with respect to $p$. So we shall denote it by $\deg(G)$.

We shall also assume that $\det L_i>0$ for any $i=1,\ldots,k$.

  \begin{lemma}\label{thm.1}
    If $G$ is as above, then $\deg (G)>0$.  In particular, if there
    exists $q \neq 0$ such that its preimage $G^{-1}(q)$ is a
    singleton that belongs to at most two of the convex polyhedral
    cones $S_i$, then $\deg (G)=1$. 
  \end{lemma}

  \begin{proof}
    Let us assume in addition that $q\notin\cup_{i=1}^kG\big(\partial
    S_i\big)$.  Observe that the set $\cup_{i=1}^kG\big(\partial
    S_i\big)$ is nowhere dense hence $A:=G(S_1)\setminus\cup_{i=1}^k
    G\big(\partial S_i \big)$ is non-empty.
    
    Take $x\in A$ and observe that if $y\in G^{-1}(x)$ then $y \notin
    \cup_{i=1}^k \partial S_i$. Thus
    \begin{equation}\label{eq.1}
      \deg(G)=\sum_{y\in G^{-1}(x)}\mathrm{sign}\,\det \ud G(y)
      =\# G^{-1}(x).
    \end{equation}
    Since $G^{-1}(x)\neq\emptyset$, $\deg(G)>0$.  The second part of
    the assertion follows taking $x=q$ in \eqref{eq.1}.

    Let us now remove the additional assumption. Let $\{p\}=G^{-1}(q)$
    be such that $p\in\partial S_i\cap\partial S_j$ for some $i\neq
    j$.  Observe that by assumption $p\neq 0$ does not belong to any
    cone $\partial S_s$ for $s\notin\{ i,j\}$. Thus one can find a
    neighborhood $V$ of $p$, with $V\subset \inter (S_i\cup S_j
    \setminus\{0\})$. By the excision property of the topological
    degree $\deg(G)=\deg(G,V,p)$. Let $\mathcal{L}_{L_iL_j}$ be a map
    as in Lemma \ref{lem:lemalg}.  Observe that, the assumption on the
    signs of the determinants of $L_i$ and $L_j$ imply that
    $\mathcal{L}_{L_iL_j}$ is orientation preserving.  Also notice
    that $\mathcal{L}_{L_iL_j}|_{\partial V}=G|_{\partial V}$. The
    multiplicativity, excision and boundary dependence properties of
    the degree yield
    $1=\deg(\mathcal{L}_{L_iL_j})=\deg(\mathcal{L}_{L_iL_j},V,p)=\deg(G,V,p)$.
    Thus, $\deg(G)=1$, as claimed.
%
  \end{proof}

  \subsection{Piecewise differentiable functions}
  $\qquad$


\begin{lemma}\label{lem:lemclin}
  Let $A$ and $B$ be linear endomorphisms of $\R^n$. Assume that for
  some $v\in\R^n \setminus \{0 \}$, $A$ and $B$ coincide on the space
  $\{x\in\R^n:\langle x,v\rangle =0\}$. Then
  \[
  \det\big(tA+(1-t)B\big)=t\det A+(1-t)\det B \quad \forall t \in \R.
  \]
\end{lemma}

\begin{proof}
  We can, without loss of generality, assume that $|v|=1$. We can
  choose vectors $w_2,\ldots,w_n\in\R^n\setminus\{0\}$ such that
  $v,w_2,\ldots,w_n$ is an orthonormal basis of $\R^n$. In this basis,
  for $t\in[0,1]$ we can represent the operator $tA+(1-t)B$ in the
  following matrix form:
  \[
  \begin{pmatrix}
    ta_{11}+(1-t)b_{11} & a_{12} &\ldots & a_{1n}\\
    \vdots & \vdots & & \vdots\\
    ta_{n1}+(1-t)b_{n1} & a_{n2} &\ldots & a_{nn}
  \end{pmatrix}
  =
  \begin{pmatrix}
    ta_{11}+(1-t)b_{11} & b_{12} &\ldots & b_{1n}\\
    \vdots & \vdots & & \vdots\\
    ta_{n1}+(1-t)b_{n1} & b_{n2} &\ldots & b_{nn}
  \end{pmatrix}
  \]
  Thus,
  \begin{align*}
    \det\big( tA+(1-t)B\big)
    =&\sum_{i=1}^n (-1)^{i+1}\big(ta_{i1}+(1-t)b_{i1}\big)\det A_{i1}\\
    =&\sum_{i=1}^n (-1)^{i+1}\big(ta_{i1}+(1-t)b_{i1}\big)\det B_{i1}
  \end{align*}
  where $A_{i1}$ and $B_{i1}$ represent the $(i1)$-th cofactor of $A$
  and $B$ respectively.  Clearly, $A_{i1}=B_{i1}$ for
  $i=1,\ldots,n$. Hence, we have
  \begin{multline*}
    \det\big( tA+(1-t)B\big)=\sum_{i=1}^n (-1)^{i+1}ta_{i1}\det A_{i1}+\\
    +\sum_{i=1}^n (-1)^{i+1}(1-t)b_{i1} \det B_{i1}= t\det A+(1-t)\det
    B
  \end{multline*}
  as claimed in the lemma.
\end{proof}

Lemmas \ref{lem:lemalg} and \ref{lem:lemclin} imply the following
fact:

\begin{lemma}\label{lem:mono}
  Let $A$ and $B$ be linear automorphisms of $\R^n$. Assume that for
  some $v\in\R^n \setminus \{0 \}$, $A$ and $B$ coincide on the space
  $\{x\in\R^n:\langle x,v\rangle =0\}$.  Assume that the map
  $\mathcal{L}_{AB}$ defined by $x\mapsto Ax$ if $\langle
  x,v\rangle\geq 0$, and by $x\mapsto Bx$ if $\langle x,v\rangle\leq
  0$, is a homeomorphism.  Then,
  $\det(A)\cdot\det\big(tA+(1-t)B\big)>0$ for any $t \in [0,1]$.
\end{lemma}

Let $\sigma_1,\ldots,\sigma_r$ be a family of $C^1$-regular pairwise
transversal hyper-surfaces in $\R^n$ with $x_0\in\cap_{i=1}^r\sigma_i$
and let $U\subset\R^n$ be an open and bounded neighborhood of
$x_0$. Clearly, if $U$ is sufficiently small,
$U\setminus\cup_{i=1}^r\sigma_i$ is partitioned into a finite number
of open sets $U_1,\ldots,U_k$.

Let $f:\overline{U}\to\R^n$ be a continuous map such that there exist
$f_1,\ldots,f_k\in C^1(\overline U)$ with the property that
\begin{equation}
  f(x)=f_i(x), \qquad x\in \overline U_i,\label{eq:fpez}
\end{equation}
with $f_i(x) = f_j(x)$ for any $x \in \overline U_i \cap \overline
U_j$.  Notice that such a function is $PC^1(\overline U)$ (see
e.g.~\cite{KS94} for a definition), and Lipschitz continuous in $U$.

Let $S_1,\ldots,S_k$ be the tangent cones (in the sense of Boulingand)
at $x_0$ to the sets $U_1,\ldots, U_k$, (by the transversality
assumption on the hyper-surfaces $\sigma_i$ each $S_i$ is a convex
polyhedral cone with non empty interior) and assume $\ud f_i(x_0)x =
\ud f_j(x_0)x$ for any $x \in S_i \cap S_j$.  Define
\begin{equation}\label{eq:Flin}
  F(x)= \ud f_i(x_0)x \qquad x\in S_i.
\end{equation}
so that $F$ is a continuous piecewise linear map (compare
\cite{KS94}).

One can see that $f$ is Bouligand differentiable and that its
B-derivative is the map $F$ (compare \cite{KS94, PR96}). Let $y_0 :=
f(x_0)$.  There exists a continuous function $\varepsilon$, with
$\varepsilon(0)=0$, such that $f(x)= y_0 +F(x-x_0)+|x-x_0|\varepsilon
(x-x_0)$.
\begin{lemma}\label{lemma.2}
  Let $f$ and $F$ be as in \eqref{eq:fpez}-\eqref{eq:Flin}, and
  assume that $\det \ud f_i(x_0) > 0 $ for all $i=1,\ldots,k$. Then
  there exists $\rho>0$ such that
  $\deg\big(f,B(x_0,\rho),y_0\big)=\deg\big(F,B(0,\rho),0\big)$.  In
  particular, $\deg\big(f,B(x_0,\rho),y_0\big)=\deg(F)$.
\end{lemma}

  \begin{proof}
    Consider the homotopy $ H(x,\lambda)=F(x-x_0)+\lambda
    \abs{x-x_0}\varepsilon(x-x_0)$, $\lambda\in [0,1] $ and observe
    that
    \[
    m := \inf\{\abs{F(v)} \colon \abs{v}=1 \}
    =\min_{i=1,\ldots,k}\|df_i\|>0.
    \]
    Thus,
    \[
    \abs{H(x,\lambda)} \geq \big( m - \abs{\varepsilon(x-x_0)}
    \big)\abs{x-x_0}.
    \]
    This shows that in a conveniently small ball centered at $x_0$,
    homotopy $H$ is admissible. The assertion follows from the
    homotopy invariance property of the degree.
  \end{proof}

  \begin{theorem} \label{thm:topo} Let $f$ and $F$ be as in
    \eqref{eq:fpez}-\eqref{eq:Flin} and assume $\det\ud f_i(x_0) >
    0$.  Assume also that there exists $p\in\R^n$ whose pre-image
    belongs to at most two of the convex polyhedral cones $S_i$ and
    such that $F^{-1}(p)$ is a singleton. Then $f$ is a Lipschitzian
    homeomorphism in a sufficiently small neighborhood of $x_0$.
  \end{theorem}
  
  \begin{proof}
    From Lemmas \ref{thm.1}-\ref{lemma.2}, it follows that
    $\deg(f,B(x_0,\rho),y_0)=1$ for sufficiently small $\rho>0$.  By
    Theorem 4 in \cite{PR96}, we immediately obtain the assertion.
  \end{proof}

\end{document}